\documentclass[11pt]{amsart}

\usepackage{amsmath,amsthm,amsfonts,amssymb,latexsym,mathrsfs,cite,cases,url,enumerate}

\usepackage{color}
\usepackage[pagebackref]{hyperref}
\usepackage[capitalize]{cleveref}
\usepackage{enumitem}
\usepackage{verbatim}
\usepackage{ifthen}
\usepackage{tikz}
\usepackage{doi}
\usepackage{float}

\usetikzlibrary{calc}
\usetikzlibrary{arrows,automata} 
\usetikzlibrary{matrix,backgrounds}

\def\NN{\mathbb{N}}

\def\Dyck{\mathrm{Dyck}}
\def\half{\mathrm{half}}
\def\SDyck{\mathrm{SDyck}}
\def\rk{\mathrm{rank}}
\def\ff{\mathfrak{f}}
\def\fh{\mathfrak{h}}

\def\ZZ{\mathbb{Z}}
\def\fD{\mathfrak{D}}
\def\rev{\mathrm{rev}}
\def\suc{\mathrm{suc}}
\def\cF{\mathcal{F}}
\allowdisplaybreaks

  \makeatletter
  \newcommand{\numberlike}[2]{%
     \expandafter\def\csname c@#1\endcsname{%
         \expandafter\csname c@#2\endcsname}%
  }
  \makeatother

  \def\DefaultNumberTheoremWithin{section}

  \theoremstyle{plain}
  \newtheorem{lemma}{Lemma}
     \numberwithin{lemma}{\DefaultNumberTheoremWithin}
     \labelformat{lemma}{Lemma~#1}
  \newtheorem{theorem}{Theorem}
     \numberwithin{theorem}{\DefaultNumberTheoremWithin}
     \numberlike{theorem}{lemma}
     \labelformat{theorem}{Theorem~#1}
  \newtheorem*{maintheorem*}{Corollary of Main Theorem}
  \newtheorem{corollary}{Corollary}
     \numberwithin{corollary}{\DefaultNumberTheoremWithin}
     \numberlike{corollary}{lemma}
     \labelformat{corollary}{Corollary~#1}
  \newtheorem{proposition}{Proposition}
     \numberwithin{proposition}{\DefaultNumberTheoremWithin}
     \numberlike{proposition}{lemma}
     \labelformat{proposition}{Proposition~#1}
  \newtheorem{conjecture}{Conjecture}
     \numberwithin{conjecture}{\DefaultNumberTheoremWithin}
     \numberlike{conjecture}{lemma}
     \labelformat{conjecture}{Conjecture~#1}

  \theoremstyle{definition}
  
     \numberwithin{definition}{\DefaultNumberTheoremWithin}
     \numberlike{definition}{lemma}
     \labelformat{definition}{Definition~#1}
  \newtheorem{question}{Question}
     \numberwithin{question}{\DefaultNumberTheoremWithin}
     \numberlike{question}{lemma}
     \labelformat{question}{Question~#1}
  
     \numberwithin{problem}{\DefaultNumberTheoremWithin}
     \numberlike{problem}{lemma}
     \labelformat{problem}{Problem~#1}
   
  \theoremstyle{remark}
  
     \numberwithin{remark}{\DefaultNumberTheoremWithin}
     \numberlike{remark}{lemma}
     \labelformat{remark}{Remark~#1}
  \newtheorem{example}{Example}
     \numberwithin{example}{\DefaultNumberTheoremWithin}
     \numberlike{example}{lemma}
     \labelformat{example}{Example~#1}

     \numberwithin{claim}{\DefaultNumberTheoremWithin}
     \numberlike{claim}{lemma}
     \labelformat{claim}{Claim~#1}
\labelformat{figure}{Figure~#1}
\labelformat{chapter}{Chapter~#1}
\labelformat{appendix}{Appendix~#1}
\labelformat{section}{Section~#1}
\labelformat{subsection}{Subsection~#1}

\title[Simplicial Complexes of Antichains and Combinatorics of Dyck Paths]{Simplicial Complexes of Antichains in Root Posets and Related Combinatorics of Dyck Paths}

\author{Lili Mu} 
\address{School of Mathematics and Statistics, Jiangsu Normal University, Xuzhou 221116, PR China} 
\email{lilimu@jsnu.edu.cn}

\author{Volkmar Welker}
\address{Philipps-Universit\"at Marburg, Fachbereich Mathematik und Informatik, 35032 Marburg, Germany}
\email{welker@mathematik.uni-marburg.de}

\begin{document}

\begin{abstract}
For a crystallographic root system $\fD$ we consider the simplicial complex $\Delta_{\fD}$ of all antichains in the
root poset of $\fD$. We show that $\Delta_{\fD}$ is shellable if
and only if $\fD$ is $A_n$, $B_n$, $D_3$ or $G_2$. 
Since antichains in types $A_n$ and $B_n$ can be identified with 
Dyck paths and symmetric Dyck paths, respectively, this yields a simplicial complex on Dyck paths. Indeed, in type $A_n$, shellability can be extended to rational Dyck paths. The $f$- and $h$-triangles then 
yield statistics on (symmetric/rational) Dyck paths. We determine these statistics for $A_n$ and $B_n$ and leave the case of rational Dyck paths as an open problem.
\end{abstract}

\maketitle

\section{Dyck complexes through Root Poset Antichains}

Let $\fD$ be a crystallographic Dynkin type and $\Phi(\fD)$ be the crystallographic root system of type $\fD$. 
We identify $\fD$ with the corresponding set of simple roots of
$\Phi(\fD)$ and denote by
$\Phi_+(\fD) = \Phi(\fD)
\cap \NN \fD$
the corresponding set of positive roots (for definitions see for example \cite{BB05}).
The root poset of type $\fD$ is the partially ordered set on 
$\Phi_+(\fD)$ where $\alpha \leq \beta$ if and only if
$\beta - \alpha \in \NN \fD$. See \ref{fig:a6} and \ref{fig:b3} for the
root posets of type $A_5$ and $B_3$, respectively. 

We set $$\Delta_\fD = \{\, A \subseteq \Phi_+(\fD)~:~A \text{ is an antichain } \,\}.$$
Since every subset of an antichain is again an antichain it follows that
$\Delta_\fD$ is a simplicial complex. For the Dynkin type $A_n$, the complex
$\Delta_\fD$ has a simple interpretation as a simplicial complex whose simplices are identified with the Dyck paths of semilength $n$ (see \ref{fig:dyck} and the discussion in \ref{sec:a}). An analogous identification with symmetric Dyck paths exists in type
$B$ (see \ref{fig:b3} and the discussion in \ref{sec:b}). For this reason, we call $\Delta_\fD$ the Dyck complex of 
type $\fD$. 
We refer the reader to \cite{Arm09} or \cite{Rin16} for further background on root posets and related topics.

These identifications indicate that a lot of Catalan combinatorics may be encoded in $\Delta_\fD$.
In this work we want to unravel some of these connections.
To do so, we first provide insight into the geometric combinatorial structure of 
$\Delta_\fD$. This will allow us to reveal some of the Catalan combinatorics  of $\Delta_\fD$
through its $f$- and $h$-triangles in \ref{thm:typeA-f}, \ref{thm:typeA-h}, \ref{thm:typeB-f}, and \ref{thm:typeB-h}. 

\begin{theorem} \label{thm:antichain}
    Let $\fD$ be one of the Dynkin types $A_n$, $B_n$, $D_3$, or $G_2$. Then the simplicial complex $\Delta_\fD$ is shellable. For the Dynkin types $D_n$ with $n\geq 4$, $F_4$, $E_6$, $E_7$, and $E_8$, the simplicial complex $\Delta_\fD$ is not shellable.
\end{theorem}

The proofs for types $A_n$, $B_n$ and $D_n$ can be found in \ref{sec:a}, \ref{sec:b} and \ref{sec:d}, respectively. The exceptional types are
treated in \ref{sec:exc}. Note that the positive cases in \ref{thm:antichain} are precisely those types whose root posets contain no interval isomorphic to the Boolean lattice of rank $3$. Equivalently, these are the root posets which in the 
sense of \cite{Rin18} have only one cubical layer.

Using the correspondence between Dyck paths and antichains in the root poset in type $A_n$, 
we can define, in type $A_n$, a more general simplicial complex $\Delta_{\Dyck_n}^{(a,b)}$ on the set of
rational or $(a,b)$-Dyck paths of $(a,b)$-semilength $n$
for coprime natural numbers $(a,b)$ (see e.g. \cite{ALW15, BGSX16, CDH16}). 
The case $(a,b)= (1,1)$ is the
case of classical Dyck paths or equivalently antichains in the root poset of type $A_n$.
We write $\Delta_{\Dyck_n}$ for $\Delta_{\Dyck_n}^{(1,1)}$ which is isomorphic to 
$\Delta_{A_n}$. In \ref{sec:a} we prove the following result which implies 
\ref{thm:antichain} in type $A_n$.

\begin{theorem}\label{thm:shellable}
    Let $(a,b)$ be a pair of coprime natural numbers. Then 
    $\Delta_{\Dyck_n}^{(a,b)}$ is a shellable simplicial complex. In particular, 
    $\Delta_{\Dyck_n} = \Delta_{A_n}$ is
    a shellable simplicial complex.
\end{theorem}

A correspondence analogous to the one in type $A_n$ identifies the
faces of $\Delta_{B_n}$ with symmetric Dyck paths. 

The sum of the entries of the $f$-triangle (see \cite{BjoWac96}) of $\Delta_{A_n} = \Delta_{\Dyck_n}$ is the number of Dyck paths and for $\Delta_{B_n}$ the sum counts symmetric Dyck paths. We work out the statistics provided by the $f$-triangle on (symmetric) Dyck paths in \ref{thm:typeA-f} and \ref{thm:typeB-f}. By \cite{BjoWac96} the $h$-triangle of a shellable simplicial complex provides a statistic on the set of facets of the complex. 
We work out these statistics for type $A_n$ and $B_n$ in \ref{thm:typeA-h} and \ref{thm:typeB-h}. Even though there appear to be some relations to known statistics all those statistics seem to be new.

Indeed, our original motivation for defining and studying $\Delta_{\fD}$ 
is its $f$-vector and a
question posed by Bell and Skandera \cite{BS07}:

\begin{question} \label{que:bs}
Let 
$1+f_0t+\cdots +f_{d-1}t^d \in \NN[t]$ be a polynomial with
only real roots. Is there a simplicial complex 
$\Delta$ whose $f$-vector is $(1,f_0,\ldots, f_{d-1})$? 
\end{question}

In recent work \cite{MW26}, the authors obtained several partial
affirmative answers to this question. 
They also showed that several classical real-rooted
polynomials, including the Eulerian polynomials, the Stirling
polynomials of the second kind, and the derangement polynomials, are
$f$-polynomials of simplicial complexes.

Another classical family of real-rooted
polynomials arises from the Narayana numbers defined as
$
N_{n,k}
=
\frac{1}{n}\binom{n}{k}\binom{n}{k-1},
 1\leq k\leq n.
$
The corresponding Narayana polynomial
$
N_n(t)=\sum_{k=1}^{n}N_{n,k}t^{k-1}
$
is real-rooted; see \cite[Theorem~5.3.1]{Bre89}. Since
$N_{n,1}=1$, the polynomial $N_n(t)$ is an instance of the
Bell--Skandera question. It is an immediate consequence of the definition and well-known facts about antichains in the root poset (see \cite{Rin18}) that $\Delta_{\Dyck_n}$ answers the question positively for Narayana numbers.
Thus, our construction not only verifies the required
$f$-vector property but also provides an explicit combinatorial
realization in terms of Dyck paths. Narayana numbers and polynomials are defined for every type 
(see e.g. \cite{Arm09} or \cite{Rin16}). 
For type $B_n$, the real-rootedness of their generating polynomial follows from 
a result by Simion \cite{Sim84}, for $D_n$ it is a result of Br\"anden \cite{Bra04} and
for the exceptional types it was first mentioned in \cite[Section 5.2]{RW05}. The latter also contains
a discussion of all types. 
By well-known facts about antichains in the root poset (see \cite{Rin18}) the type $\fD$ Narayana numbers are the entries in the $f$-vector of $\Delta_{\fD}$. This is what initially motivated us to  study the simplicial complexes
$\Delta_{\fD}$.

The paper is structured as follows. In \ref{sec:basic} we provide the basic definitions around
simplicial complexes, their $f$- and $h$-triangles and shellability. In \ref{sec:a} we prove
\ref{thm:shellable} which implies \ref{thm:antichain} in type $A_n$. We also give explicit combinatorial interpretations of the $f$- and
 $h$-triangle for $\Delta_{A_n}$. We ask for an extension of these interpretations to all
 rational Dyck paths. In the Appendix, we provide $f$- and $h$-triangles of rational Dyck paths counted by Fuss-Catalan numbers for small parameters.
 In \ref{sec:b}, we prove \ref{thm:antichain} for type $B_n$ and
also provide combinatorial interpretations of the $f$- and $h$-triangles.
In \ref{sec:d}, we cover the type $D_n$ case of \ref{thm:antichain}. In \ref{sec:exc}, we then
treat all exceptional cases of \ref{thm:antichain}.

All calculations of the homology, the $f$-triangles, and the $h$-triangles of the $\Delta_{\fD}$ were performed with the help of Sage \cite{sagemath}.

\section{Basic notation and results on simplicial complexes} \label{sec:basic}

In this section we introduce some basic notions and results around simplicial complexes
(see \cite{BjoWac96} for more details).

A simplicial complex $\Delta$ on a (finite) ground set $\Omega$ is a
subset of $2^\Omega$ such that $\tau \subseteq \sigma \in \Delta$ implies $\tau \in \Delta$.
We call $\sigma
\in  \Delta$ a face of $\Delta$ and an inclusion-wise maximal face of $\Delta$ is called a maximal face. The dimension $\dim (\sigma)$ of a face is 
$|\sigma| -1$. The dimension 
$\dim(\Delta)$ of $\Delta$ is the maximal dimension of one of its faces. 
The vector $(f_{-1},\ldots, f_{d-1})$ for
$d-1 = \dim (\Delta)$ and $f_i = \big|\,\{
\,\sigma \in \Delta~:~\dim \sigma = i\,\}\,\big|$ is 
called the $f$-vector of $\Delta$. 
The degree $\deg(\sigma)$ of a face
$\sigma \in \Delta$ is 
$\max \{\, \dim(\tau)+1~:~
\tau \in \Delta \text{ and } \sigma \subseteq \tau\,\}$. 
Let $f_{ij}$ be the number of faces
$\sigma$ of $\Delta$ with $\dim(\sigma) +1 
= |\sigma| = j$ and $\deg(\sigma) = i$.
Then the array $(f_{ij})_{0 \leq j \leq i \leq d}$
is called the $f$-triangle of $\Delta$.
It follows from the
definitions that 
$f_{j-1} = f_{d,j} + \cdots + f_{j,j}$.
The single-variable $f$-polynomial $\ff^\Delta(t)$ of a simplicial complex $\Delta$ 
of dimension $d-1$ is
defined as $\ff^ \Delta(t) = \sum_{i=-1}^{d} f_{i}t^ {d-i}$
which by the above identity satisfies
$\ff^ \Delta(t) = \frac{1}{t^{d+1}} \ff^\Delta(t,\frac{1}{t})$
for the two-variable $f$-polynomial $\ff^ \Delta(x,y)= \sum_{0 \leq j \leq i} f_{ij}x^ iy^{i-j}$.
Note that \ref{que:bs} applies to the polynomial $t^{d+1} \ff^\Delta(\frac{1}{t})$. 
The coefficients $(h_{ij})_{0 \leq j \leq i}$ of the two-variable $h$-polynomial
$\fh^\Delta(x,y) = \sum_{0 \leq j \leq i} h_{ij} x^iy^ {i-j} = \ff^ \Delta(x,y-1)$
are called the $h$-triangle of $\Delta$.
The coefficient vector $(h_0,\ldots, h_d)$ of the single variable $h$-polynomial 
$\fh^ \Delta(t) = \sum_{i=0}^ d h_it^{d-i} = \frac{1}{(t-1)^ {d+1}}\fh^ \Delta(\frac{1}{t-1},t)$
is called the $h$-vector of $\Delta$ and $\fh^\Delta(t)$ the single-variable $h$-polynomial
of $\Delta$.

By the above formulas, the $h$-triangle and $h$-polynomial can in principle be
calculated from the $f$-triangle and $f$-polynomial. Nevertheless, the 
transformation formulas are often hard to interpret. For ''nice'' simplicial 
complexes, the following concepts
allow for a more direct approach.

   A simplicial complex $\Delta$ is called {\it shellable} if there is a linear
  order $\prec$ on the set of its maximal faces such that 
  if $\sigma \prec \tau$ then there is $\gamma \prec \tau$ such that
  $$\sigma \cap \tau \subseteq \gamma \cap \tau = \tau \setminus \{x\}$$
  for some $x \in \tau$. The order $\prec$ is then called a shelling order.
For a shelling order $\prec$ and $\sigma$ a maximal face of $\Delta$ we define the
restriction 
  \[
R(\sigma) = \big\{\,x \in \sigma ~:~ \sigma \setminus\{x\} \subseteq \sigma' \text{ for some maximal face } \sigma' \in \Delta \text{ with } \sigma' \prec \sigma\,\big\}.
\]

 \begin{proposition}[Theorem 3.4 \cite{BjoWac96}] \label{prop:shelh}
 Let $\Delta$ be a shellable simplicial complex with shelling order $\prec$. Then for $0 \leq j  \leq i$ 
 $$h_{ij} = \Big|\, \big\{ \,\sigma \in \Delta\,:\,\sigma \text{ maximal face }, \,|\sigma| = i \text { and } |R(\sigma)| = j\,\big\}\,\Big|.$$
\end{proposition}

\section{Type \texorpdfstring{$A_n$}~: classical and \texorpdfstring{$(a,b)$}--Dyck paths} \label{sec:a} 

Through the following example, 
we explain the well-known fact that antichains in $\Phi_+(A_n)$ 
correspond to classical Dyck paths. Before we get to the example, we
recall the basic definitions around Dyck words.

For a word $w = w_1\cdots w_n$ in the alphabet $\{D, U\}$
we write $|w|$ for its length $n$. We set $D_i(w) = \{j~:~j \leq i , w_j = D\}$
and $U_i(w) = \{j~:~j \leq i , w_j = U\}$.
A Dyck path or Dyck word of semilength $n$ is a word $w = w_1\cdots w_{2n}$ of
length $|w| = 2n$ in the alphabet $\{U, D\}$ such that 
for all $1 \leq i < 2n$ we have $|D_i(w)| \leq |U_i(w)|$ and
$|D_{2n}(w)| = |U_{2n}(w)|$.

\begin{example}
   Let $\fD = A_{n}$ with root system 
   $\Phi(A_{n}) = \big\{\,\pm(e_i-e_j)~:~ 1 \leq i <j \leq n+1\,\big\}$.
   We choose the 
   simple roots $e_1-e_2, e_2-e_3,
   \ldots, e_{n}-e_{n+1}$. Then the positive roots
   are $\Phi_+(\fD) = \{ e_i - e_j \,:\,
   1 \leq i < j \leq n+1\}$.
   It is easily seen that 
   $e_i - e_j \leq e_{i'} - e_{j'}$ if and only if
   $i' \leq i < j \leq j'$. See \ref{fig:a6} for the
   Hasse diagram of the root poset for $A_5$.
   
    \begin{figure}[ht!]
\centering
\begin{tikzpicture}

\def\n{5} 
\pgfmathtruncatemacro{\N}{\n+1} 

\foreach \i in {1,...,\n} {
  \foreach \j in {\the\numexpr\i+1\relax,...,\N} {
    \pgfmathtruncatemacro{\h}{\j-\i}
    \node (r-\i-\j) at (2*\i+\h,\h) {};
    \fill (r-\i-\j) circle (2.5pt) node[right] {$~~e_\i-e_\j$};
  }
}

\foreach \i in {1,...,\n} {
  \foreach \j in {\the\numexpr\i+1\relax,...,\N} {
    \ifnum\j<\N
      \draw (r-\i-\j) -- (r-\i-\the\numexpr\j+1\relax);
    \fi
    \ifnum\i>1
      \draw (r-\i-\j) -- (r-\the\numexpr\i-1\relax-\j);
    \fi
  }
}

\end{tikzpicture}
\caption{Root Poset of Type $A_5$} \label{fig:a6}
\end{figure}
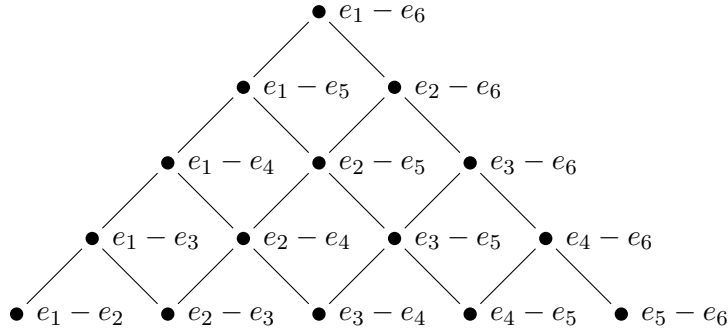

\ref{fig:dyck} visualizes the 
bijection between antichains in the root poset of
$A_n$ and Dyck paths of semilength $n+1$ in the case $n = 5$ and
for the antichain $\{e_2-e_4,e_4-e_5\}$.
   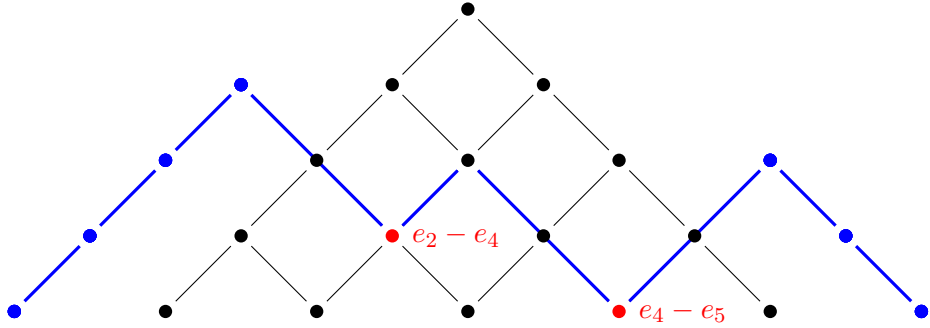
\begin{figure}[ht!]
\centering

\begin{tikzpicture}

\def\n{5} 
\pgfmathtruncatemacro{\N}{\n+1}

\foreach \i in {1,...,\n} {
  \foreach \j in {\the\numexpr\i+1\relax,...,\N} {
    \ifnum\j<\N
      \draw (r-\i-\j) -- (r-\i-\the\numexpr\j+1\relax);
    \fi
    \ifnum\i>1
      \draw (r-\i-\j) -- (r-\the\numexpr\i-1\relax-\j);
    \fi
  }
}
\node (A) at (1,1) {};
 \node (B) at (2,2) {};
   \node (C) at (3,3) {};
      \node (D) at (4,4) {};
     \node (E) at (11,3) {};
      \node (F) at (12,2) {};
      \node (G) at (13,1) {};

\draw[blue, very thick] (A) -- (B) -- (C) -- (D);

\draw[blue,very thick] (E) -- (F) -- (G);

    \draw[blue,very thick] (D) -- (r-2-4);
\draw[blue,very thick] (r-2-4) -- (r-2-5);
\draw[blue,very thick] (r-2-5) -- (r-4-5);
\draw[blue,very thick] (r-4-5) -- (E);

\foreach \i in {1,...,\n} {
  \foreach \j in {\the\numexpr\i+1\relax,...,\N} {
    \pgfmathtruncatemacro{\h}{\j-\i}
    \node (r-\i-\j) at (2*\i+\h,\h) {};
    \ifnum \i = 2 
      {\ifnum \j = 4
        \fill[red] (r-\i-\j) circle (2.5pt) node[right]
          {$~~e_2-e_4$};
        \else 
        \fill (r-\i-\j) circle (2.5pt);
        \fi}
    \else
        {\ifnum \i = 4
          {\ifnum \j = 5
             \fill[red] (r-\i-\j) circle (2.5pt) node[right]
          {$~~e_4-e_5$};; 
           \else 
             \fill (r-\i-\j) circle (2.5pt);
           \fi}
        \else
         \fill (r-\i-\j) circle (2.5pt);
       \fi}
       \fi
  }
  
  \fill[blue] (A) circle (2.5pt);
   \fill[blue] (B) circle (2.5pt);
   \fill[blue] (C) circle (2.5pt);
   \fill[blue] (D) circle (2.5pt);
   \fill[blue] (E) circle (2.5pt);
   \fill[blue] (F) circle (2.5pt);
   \fill[blue] (G) circle (2.5pt);
}

\end{tikzpicture}
\caption{The Dyck Path of Semilength $6$ corresponding to the Antichain 
$\{e_2-e_4,e_4-e_5\}$ in the Root Poset of Type $A_5$} \label{fig:dyck}
\end{figure}
\end{example}

Indeed, we will consider the more general setting of $(a,b)$-Dyck paths (see \cite{ARW13}).
For two positive coprime integers $a,b$
we say that $w = w_1\cdots w_r$ is an $(a,b)$-Dyck path or $(a,b)$-Dyck word if
$b\cdot |D_i(w)| \leq a \cdot |U_i(w)|$ for $1 \leq i < r$ and
$b \cdot |D_{r}(w)| = a \cdot |U_{r}(w)|$.
It follows that $r = |U_r(w)|\,(1+\frac{a}{b})$.
Since $a$ and $b$ are coprime, it also follows that $|U_r(w)|$ is an integer multiple of $b$, say $nb$, and $|D_r(w)|$ is an integer multiple of $a$ for the same
number $n$. In particular, $r = n(a+b)$.
We call the integer $n$ determined by $nb= |U_r(w)|$ the $(a,b)$-length of the $(a,b)$-Dyck path $w$.
An $(a,b)$-Dyck path $w$ is also called a rational Dyck path.
For $(a,b) = (1,1)$ an $(a,b)$-Dyck path is a Dyck path in the
usual sense and the $(1,1)$-length is the semilength of the Dyck path. 

We denote by $\Dyck_n^ {(a,b)}$ the set of all
$(a,b)$-Dyck paths of $(a,b)$-semilength $n$. We write $\Dyck_n$ for the classical $\Dyck_n^{(1,1)}$.
A valley of the Dyck word $w \in \Dyck_n^{(a,b)}$ is some $1 \leq i < r$ such that
$w_i = D$ and $w_{i+1} = U$.
We often identify the valley with its position $(\,i,U_i(w) - D_i(w)\,)$ 
and call $U_i(w)-D_i(w)$
the height of the valley. We denote by
$V(w)$ the set of all positions of the valleys of $w$. The first lemma is an immediate consequence of
the definitions.

\begin{lemma}
    Let $a,b$ be two positive coprime integers. 
    The map $V : \Dyck_n^{(a,b)} \rightarrow 2^{\NN \times \NN}$ sending $w \in \Dyck_n^{(a,b)}$ to $V(w)$  is injective. 
\end{lemma}

\begin{proof}
Let the valleys of $w$ be ordered from left to right. The portion
of $w$ between two consecutive valleys contains no other valley
and hence consists of a sequence of $U$-steps followed by a
sequence of $D$-steps. The same holds for the portions preceding
the first valley and following the last valley. The numbers of
$U$- and $D$-steps in each such portion are uniquely determined
by the coordinates of its endpoints. Since all paths in
$\Dyck_n^{(a,b)}$ have the same initial and terminal points, the
set $V(w)$ uniquely determines $w$. Thus, the map $V$ is
injective.
\end{proof}

We now define $\Delta_{\Dyck_n}^{(a,b)}$ as $\{ \,V(w)~:~w \in \Dyck_n^{(a,b)}\,\}$.
Following our convention, we write $\Delta_{\Dyck_n}$ for $\Delta_{\Dyck_n}^{(1,1)}$.

\begin{lemma}
    For two positive coprime integers $a,b$ the set $\Delta_{\Dyck_n}^{(a,b)}$ is a simplicial complex. 
    For $(a,b) = (1,1)$ its dimension is $n-2$.
\end{lemma}
\begin{proof}
    It suffices to show that if $(x,y) \in V(w)$ then there is
    $w' \in \Dyck_n^{(a,b)}$ such that $V(w') = V(w) \setminus \{(x,y)\}$. 
    We factor $w = w_1U\underbrace{D \cdots D}_d \underbrace{U\cdots U}_u D w_2$ such that $w$ reaches the valley $(x,y)$ after
        $w_1U\underbrace{D \cdots D}_d$. Set
        $w' = w_1U \underbrace{U\cdots U}_u \underbrace{D \cdots D}_d Dw_2$. Then $(x,y)$ is no longer a valley of $w'$ and all
        other valleys are preserved. It follows that 
        $V(w') = V(w) \setminus \{(x,y)\}$. The condition on $(a,b)$-Dyck paths is clearly preserved from $w$ to $w'$.

        The Dyck path
        $\underbrace{UD \cdots UD}_{2n}$
        has $n-1$ valleys and clearly
        maximizes the number of valleys
        for $(1,1)$-Dyck paths. It follows that
        the dimension of 
        $\Delta_\Dyck^{(1,1)}$ is $n-2$.
\end{proof}

Next, we classify the maximal faces of 
$\Delta_{\Dyck_n}^{(a,b)}$.

\begin{lemma} \label{lem:max}
  Let $a,b$ be two positive coprime integers and let $w \in \Dyck_n^{(a,b)}$ be an $(a,b)$-Dyck path. Then $V(w)$ is a maximal face
  of $\Delta_{\Dyck_n}^{(a,b)}$ if and only if $w$ does not 
  have a factorization $w_1UUDDw_2$ for possibly empty
  words $w_1,w_2$. 
\end{lemma}

\begin{proof} 
Suppose first that $w$ contains a factor $UUDD$. 
Thus, we may write $w$ as $w=w_1\,UUDD\,w_2.$ Define $w'=w_1\,UDUD\,w_2.$  We first verify that $w'$ is again an $(a,b)$-Dyck path. Let $r = |w|$ then we also have $|w'| = r$. By definition 
$|U_r(w)| = |U_r(w')|$ and $|D_r(w)| = |D_r(w')|$. In particular, $\frac{a}{b} = \frac{|D_r(w)|}{|U_r(w)|} = \frac{|D_r(w')|}{|U_r(w')|}$. 
Thus the $(a,b)$-Dyck path condition for $w'$ is equivalent to $a|U_i(w')|-b|D_i(w')| \geq 0$
for $0 \leq i < r$. 

Let $j$ be the length of the word $w_1UU$. By
construction, if $i \neq j$ then $U_i(w) = U_i(w')$ or equivalently $D_i(w) = D_i(w')$. 
Since $w$ is an $(a,b)$-Dyck path it follows that $a|U_i(w)|-b|D_i(w)| \geq 0$
for all $i$
and hence $a|U_i(w')|-b|D_i(w')| \geq 0$ for $i \neq j$. 
Again by construction $|U_j(w')| = |U_{j+2}(w)|-1 = |U_{j-2}(w)|+1$
and $|D_j(w')| = |D_{j+2}(w)|-1 = |D_{j-2}(w)|+1$. It follows that
\begin{align*}
     a|U_j(w')|-b|D_j(w')| & = a(|U_{j-2}(w)|+1)-b(|D_{j-2}(w)|+1)  \\
     & = (a-b)
+ a|U_{j-2}(w)|-b|D_{j-2}(w)|  \overset{w \in \Dyck_n^ {(a,b)}}{\geq} a-b\end{align*}
and
\begin{align*} a|U_j(w')|-b|D_j(w')| & = a(|U_{j+2}(w)|-1)-b(|D_{j+2}(w)|-1) \\
& = (b-a)
+ a|U_{j+2}(w)|-b|D_{j+2}(w)| \overset{w \in \Dyck_n^ {(a,b)}}{\geq} b-a.\end{align*}
As a consequence we have that $a|U_j(w')|-b|D_j(w')| \geq 0$ and hence $w' \in \Dyck_n^ {(a,b)}$.

The replacement of $UUDD$ by $UDUD$ creates exactly one new valley and preserves all the other valleys. Therefore, \[ V(w')=V(w)\cup\{(x,y)\} \] for some $(x,y)\notin V(w)$. Consequently, 
$V(w)$ is not a maximal face of $\Delta_{\Dyck_n}^{(a,b)}$. 

Conversely, suppose that $V(w)$ is not a maximal face of $\Delta_{\Dyck_n}^{(a,b)}$. Since $\Delta_{\Dyck_n}^{(a,b)}$ is a simplicial complex, there exist $(x,y)\notin V(w)$ and $w'\in\Dyck_n^{(a,b)}$ such that \[ V(w')=V(w)\cup\{(x,y)\}. \] Order the valleys in $V(w)$ in ascending order by their 
first coordinate and adjoin the initial and terminal points of the path. Let $(x',y')$ and $(x'',y'')$ be the two consecutive points in this order such that $x' < x < x''$.

The portion of $w$ between $(x',y')$ and $(x'',y'')$ contains no valley. Therefore, its corresponding subword is of the form $U^u D^d$  for some positive integers $u$ and $d$. On the other hand, the portion of $w'$ between the same two points contains exactly one valley, namely $(x,y)$. Hence its corresponding subword has the form \[ U^{u_1}D^{d_1}U^{u_2}D^{d_2}, \] where $u_1,u_2,d_1,d_2\geq 1$.  Since these two subpaths have the same initial and terminal points, they contain the same total numbers of $U$-steps and $D$-steps. Thus, \[ u=u_1+u_2\geq2 \qquad\text{ and }\qquad d=d_1+d_2\geq2. \] It follows that the factor $U^u D^d$ of $w$ contains $UUDD$ as a subfactor. This completes the proof. 

\end{proof}
     
  We will call a Dyck path $w \in \Dyck_n^{(a,b)}$ a $UUDD$-avoiding Dyck path if 
  $w$ has no factorization $w=w_1UUDDw_2$. By \ref{lem:max} the $UUDD$-avoiding Dyck paths 
  $w$ are
  exactly those for which $V(w)$ is a maximal face of $\Delta_{\Dyck_n}^{(a,b)}$.

Next, we define a linear order $\prec$ on the maximal faces of $\Delta^{(a,b)}_{\Dyck_n}$.
    By the injectivity of the map $V$ it suffices to define the order on 
    $(a,b)$-Dyck words.
    Let $v$ and $w$ be two $(a,b)$-Dyck paths corresponding to maximal faces of 
    $\Delta^{(a,b)}_{\Dyck_n}$.
    Let $V(v)=\big\{\,(x_1,y_1),\dots,(x_\ell,y_\ell)\,\big\}$ and
    $V(w) = \big\{\,(x_1',y_1'),\dots,(x_{\ell'}',y_{\ell'}')\,\big\}$ be the positions of the 
    valleys of $v$ 
    and $w$. Without loss of generality we will always 
    assume that $x_1<x_2<\cdots<x_\ell$ and $x_1' < x_2'< \cdots < x_{\ell'}'$. So in principle we
    consider $V(v)$ and $V(w)$ as tuples of valley positions.  
    Define $v \prec w$ if and only if $V(v) <_{\mathrm{lex}} V(w)$ in the lexicographic order
    induced by $(x,y)<(x',y')$ if $x<x'$ or $x=x'$ and $y<y'$.

 Now we are in position to prove \ref{thm:shellable}.
 
\begin{proof}[Proof of \ref{thm:shellable}]
    Let $\prec$ be the linear order on $\Delta^{(a,b)}_{\Dyck_n}$
    defined above.

    Let $v \prec w$ be $(a,b)$-Dyck words corresponding to maximal faces
    \[V(v) = \big\{\,(x_1,y_1),\dots,(x_\ell,y_\ell)\,\big\} \text{ and } V(w) = \big\{\,(x_1',y_1'),\dots,(x_{\ell'}',y_{\ell'}')\,\big\} \] of $\Delta^{(a,b)}_{\Dyck_n}$. As usual we assume
    $x_1 < \cdots < x_\ell$ and $x_1' < \cdots < x_{\ell'}'$.
    Suppose $k$ is minimal such that $(x_k', y_k')$ is a valley in $w$ but not 
    in $v$. Let $w = w_1DUw_2$ where $|w_1| = x_k-1$. We consider the case when $w_1$ ends with $D$ and
    the case when it ends with $U$ separately. 

\medskip
    \noindent {\sf Case: \textbf{$w_1 = w_1'D$}  }

\medskip

    Then $w = w_1' D D U w_2$, which implies
    \begin{align} \label{eq:w1D}
    \frac{|D(w_1')| + 2}{|U(w_1')| + 1} & \le \frac{a}{b}.
    \end{align}
    We define $z = w_1'DUDw_2$. Next we have to show that $z$ is an $(a,b)$-Dyck path corresponding to a maximal face of $\Delta_{\Dyck_n}^ {(a,b)}$. 
    
    Clearly, since $w$ does not contain $UUDD$ the path $z$ also contains no $UUDD$ pattern.
    Thus by \ref{lem:max} if $z$ is an $(a,b)$-Dyck path then it is maximal.
    
    To verify that $z \in \Dyck^{(a,b)}_n$, it suffices to check
    \[
    \frac{|D(w_1')| + 1}{|U(w_1')| + 1} \le \frac{a}{b},
    \]
    which follows immediately from \eqref{eq:w1D}. We conclude that 
    $V(z)$ is a maximal face of $\Delta_{\Dyck_n}^{(a,b)}$. In $z$, the valley $(x_k',y_k')$ is
    replaced by $(x_k' - 1, y_k' + 1)$. Consequently, we have $V(z) <_{\text{lex}} V(w)$ and
    $V(z) \cap V(w) = V(w) \setminus \{\,(x_k', y_k')\,\}$. Therefore, $z  \prec w$ and
    \[
    V(v) \cap V(w) \subseteq V(z) \cap V(w) = V(w) \setminus \{\,(x_k', y_k')\,\}.
    \]

    \medskip
    \noindent {\sf Case: \textbf{$w_1 = w_1'U$} }

\medskip

 \noindent We first show that:
 
 \smallskip

 \noindent {\sf Claim:} $w_1$ does not factor as $w_1''DU$. 

 \smallskip 
 
 By definition $(x_k',y_k')$ is the first valley in $V(w)$ that does not belong to $V(v)$. 
 Given that $v \prec w$, all valleys in $w$ occurring before $(x_k',y_k')$ coincide with those in $v$. If $w_1=w_1''DU$ then $(x_{k-1}',y_{k-1}') = (x_k'-2,y_k') \in V(v),V(w)$. 
 But by $v \prec w$ and the choice of $(x_k,y_k)$ the $k$th valley of $w$ must come at a position 
 $(x_k,y_k) < (x_k',y_k')$ which then is impossible.

 \medskip
 
 The claim implies that we are left to consider the case $w_1 = w_1''UU$. 
    Next, we consider the following three cases for the factorization of $w_2$. 
    
\medskip

\begin{itemize}
    \item[$\longrightarrow$] $w_2 = U  w_2'$ 
\end{itemize}
 Define $z = w_1'' U D U U U w_2'$.
Since the two steps preceding $(x_k',y_k')$ are both $U$, we may assume  that there are $\ell \geq 2$ steps $U$ between the $(k-1)$st and the $k$th 
valley of $w$. That is, $w$ can be further factored as
\[
w = w_1''' D U^\ell  D  U^2  w_2'.
\]
The  path $z$ then factors accordingly as 
\[
z = w_1''' D U^{\ell-1}  D  U^3  w_2'.
\]

Because $k$ is minimal such that $(x_k',y_k')$ is a valley in $w$ and not a valley in $v$, it follows from $v \prec w$ that the first $k-1$ valleys of $w$ are all in $v$, and the $k$th valley $(x_k,y_k)$ of $v$ has $(x_k,y_k) < (x_k',y_k')$. Hence we can write $v = w_1''' D U^{\ell'}  D  \overline{w_2}$ for some $\ell' < \ell$. Since $v \in \Dyck^{(a,b)}_n$, we have
\begin{align} \label{eq:w2UU}
\frac{|D(w_1''')| + 2}{|U(w_1''')| + \ell'} \le \frac{a}{b}.
\end{align}

To verify that $z \in \Dyck^{(a,b)}_n$, it suffices to check
\[
\frac{|D(w_1''')| + 2}{|U(w_1''')| + \ell - 1} \le \frac{a}{b},
\]
which follows immediately from \eqref{eq:w2UU} using $\ell' \le \ell - 1$.

    \medskip
    
    \begin{itemize}
        \item[$\longrightarrow$]  $w_2 = DU w_2''$:
  \end{itemize}
        
        Define $z = w_1'' U D U UDU w_2''$. By a discussion analogous to the one above, one verifies that $z \in \Dyck^{(a,b)}_n$.

  \medskip
  
  \begin{itemize}
      \item[$\longrightarrow$]
    $w_2 = D D w_2''$:
    \end{itemize}
    Define $z = w_1'' U D U D U Dw_2''$.
   By $w\in \Dyck^{(a,b)}_n$, we have
\begin{align} \label{eq:w2DD}
\frac{|D(w_1'')| + 3}{|U(w_1'')| + 3} & \le \frac{a}{b}.
\end{align}
To verify $z = w_1'' U D U D U D w_2'' \in \Dyck^{(a,b)}_n$, it suffices to check
\[
\frac{|D(w_1'')| + 1}{|U(w_1'')| + 1} \le \frac{a}{b} \quad\text{ and }\,\quad 
\frac{|D(w_1'')| + 2}{|U(w_1'')| + 2} \le \frac{a}{b},
\]
which follow immediately from \eqref{eq:w2DD}, since
\[
\frac{|D(w_1'')| + 1}{|U(w_1'')| + 1} \le \frac{|D(w_1'')| + 3}{|U(w_1'')| + 3}
\quad\text{and}\quad
\frac{|D(w_1'')| + 2}{|U(w_1'')| + 2} \le \frac{|D(w_1'')| + 3}{|U(w_1'')| + 3}.
\]

It is easily verified that in all above cases, the constructed path $z$ satisfies
$z \prec w$ and that
$
V(v) \cap V(w) \subseteq V(z) \cap V(w) = V(w) \setminus \{\,(x_k,y_k)\,\}.
$

We have now shown that $\prec$ is a shelling order.
In particular, $\Delta_{\Dyck_n}^{(a,b)}$ is shellable.
    For $a=b= 1$ we get as a special case that $\Delta_{\Dyck_n}$ is shellable.
\end{proof}

Basic facts about non-pure shellable complexes (see \cite[Theorem 4.1]{BjoWac96}) now imply that:

\begin{corollary}\label{cor:wedge1}
    For coprime positive integers $a,b$ the simplicial complex $\Delta_{\Dyck_n}^{(a,b)}$ is homotopy equivalent to a wedge of spheres of 
    possibly different dimensions. 
\end{corollary}

The exact number of spheres and their dimensions can be read off from the $h$-vector. \ref{thm:typeA-f} and \ref{thm:typeA-h} study the $f$-vector and $f$-triangle, and
the $h$-vector and $h$-triangle 
of $\Delta_{\Dyck_n}^{(a,b)}$. Indeed, for a precise
description we will have to confine ourselves to $(a,b) = (1,1)$. 
Using \cite[Theorem 3.4]{BjoWac96}, for general $(a,b)$ we only have the following corollary of \ref{thm:shellable}.

\begin{corollary}\label{cor:wedge2}
    For coprime positive integers $a,b$ the 
    $h$-triangle $(h_{ij})_{1 \leq i\leq j \leq n-1}$ of
    $\Delta_{\Dyck_n}^{(a,b)}$ has only non-negative entries.
 \end{corollary}

Next, we describe the $f$-triangle of $\Delta_{\Dyck_n}$. For that 
and in 
the following considerations we write $\Dyck^ {\max}_n$ for the $UUDD$-avoiding Dyck paths in
$\Dyck_n$ or by \ref{lem:max} equivalently for the Dyck paths $w$ for which $V(w)$ is a maximal face of
$\Delta_{\Dyck_n}$.

Let $w\in\Dyck_n$ have valleys
$p_1 = (x_1,y_1),p_2=(x_2,y_2),\ldots,p_\ell=(x_\ell,y_\ell)$ listed such that
$x_1 < x_2 < \cdots < x_\ell$. 
We adjoin the starting point $p_0$ of $w$ and the terminal
point $p_{\ell+1}$ of $w$.
For $1\leq i\leq \ell+1$, let $u_i = u_i(w)$ and $d_i= d_i(w)$
denote the numbers of $U$-steps and $D$-steps, respectively, in
the portion of $w$ from $p_{i-1}$ to $p_i$.
We define 
$\mu(w)=\sum_{i=1}^{\ell+1}\min\bigl\{u_i,d_i\bigr\}-1$.
To the best of our knowledge, this invariant has not appeared in the literature. If $w$ is $UUDD$-avoiding, then $\mu(w) = \ell$ is the number of
valleys of $w$ and if $w = U^nD^n$, then $\mu(w) = n-1$. Thus $\mu(w) - \ell$ provides a measure of the fluctuation between valleys of $w$.

\begin{theorem}\label{thm:typeA-f}
 Let $(f_{ij})_{0 \leq j \leq i \leq n-1}$ be the $f$-triangle of 
    $\Delta_{\Dyck_n}$. Then 
\[
f_{ij}
=\Big|\,\bigl\{\,
w\in\Dyck_n
\,:\,
|V(w)|=j\text{ and }\mu(w)=i
\,\bigr\}\,\Big|,
\qquad
0\leq j\leq i\leq n-1.
\]
\end{theorem}

\begin{proof}
Fix $w\in\Dyck_n$ with
$V(w)=\big\{\,p_1 = (x_1,y_1),p_2=(x_1,y_2),\ldots, p_\ell=(x_\ell,y_\ell)\,\big\}$
indexed such that $x_1 < \cdots < x_\ell$. Define $p_0$, $p_{\ell+1}$ and 
$d_i,u_i$ for $1 \leq i \leq \ell+1$ as above.

Since for $1 \leq i \leq \ell+1$ there are no valleys of $w$ strictly between
$p_{i-1}$ and $p_i$, the corresponding subword of $w$ between the two valleys is
$U^{u_i}D^{d_i}$ for some $u_i, d_i > 0$.

Let $z\in\Dyck^ {\max}_n$ be a $UUDD$-avoiding Dyck path such that
$
V(w)\subseteq V(z).
$
Consider the subword $z_i$ of $z$ between the valleys $p_{i-1}$ and $p_i$.
This subword can be factored as
\[
z_i = U^{\alpha_1}D^{\beta_1}
U^{\alpha_2}D^{\beta_2}
\cdots
U^{\alpha_{s_i}}D^{\beta_{s_i}},
\]
where $\alpha_j,\beta_j > 0$ for $1 \leq j \leq s_i$. Since the
subpath defined by $z_i$ has the same starting point $p_{i-1}$ and
the same endpoint $p_{i}$ as the corresponding subpath of
$w$, we have
\[
\sum_{j=1}^{s_i}\alpha_j=u_i,
\qquad
\sum_{j=1}^{s_i}\beta_j=d_i.
\]
In particular,
\[
s_i\leq u_i
\qquad\text{and}\qquad
s_i\leq d_i,
\]
and hence
\[
s_i\leq\min\{u_i,d_i\}.
\]

The subword $z_i$ contributes $s_i-1$
valleys to $V(z) \setminus V(w)$. By $V(w) \subseteq V(z)$ 
the $\ell$ valleys 
$p_1,\ldots, p_\ell$ are inherited by $z$ from $w$. 
Therefore,
\[
|V(z)|
=
\ell+\sum_{i=1}^{\ell+1}(s_i-1)
=
\sum_{i=1}^{\ell+1}s_i-1.
\]
It follows that
\[
|V(z)|
\leq
\sum_{i=1}^{\ell+1}\min\{u_i,d_i\}-1
=
\mu(w).
\]
We conclude that,
\[
\deg\bigl(V(w)\bigr)\leq\mu(w).
\]

Next, we construct a $UUDD$-avoiding Dyck path $z$ attaining this
upper bound. For positive integers $u$ and $d$, define
\[
Q(u,d)
=
\begin{cases}
(UD)^{u-1}UD^{\,d-u+1},
   &\text{if }u\leq d,\\
U^{\,u-d+1}D(UD)^{d-1},
   &\text{if }d\leq u.
\end{cases}
\]
In either case, $Q(u,d)$ contains exactly $u$ up-steps and $d$
down-steps, and it consists of precisely
$\min\{u,d\}$
maximal blocks of the form $U^\alpha D^\beta$ for $\alpha,\beta > 0$. Moreover, each
such block has either $\alpha=1$ or $\beta=1$. Consequently,
$Q(u,d)$ contains no factor $UUDD$.

For every $i$, replace in $w$ the subword 
$U^{u_i}D^{d_i}$ between $p_{i-1}$ and $p_i$ 
by $Q(u_i,d_i)$, and denote the resulting word by
$z$. By construction, each replacement preserves the initial and terminal
points of the corresponding subpath. 

Next, we verify that $z$ is a Dyck path. Suppose that the
portion under consideration begins at height $h$.

If $u_i\leq d_i$, the initial factors $UD$ in
$Q(u_i,d_i)$ return repeatedly to height $h$, and the final
factor
$UD^{\,d_i-u_i+1}$
descends to the original terminal height
$h+u_i-d_i\geq0.$
Hence this subpath never goes below height $0$.

If $d_i\leq u_i$, the initial factor
$U^{\,u_i-d_i+1}D$
reaches the terminal height $h+u_i-d_i$, and every subsequent
factor $UD$ stays weakly above that height. Thus, this portion
also remains weakly above height $0$.
It follows that $z\in\Dyck_n$. 

Since each replacement word $Q(u_i,d_i)$ 
begins with $U$ and ends with $D$, all
$p_1,\ldots, p_\ell$ are also valleys of $z$. Therefore,
we have $V(w)\subseteq V(z).$
Furthermore, each subword $Q(u_i,d_i)$ is $UUDD$-avoiding, and the
junction between two consecutive subwords is of the form $DU$.
Hence $z$ is $UUDD$-avoiding. By
\ref{lem:max}, $V(z)$ is a maximal face of
$\Delta_{\Dyck_n}$.

Finally, the subword $Q(u_i,d_i)$ contributes
$\min\{u_i,d_i\}-1$
new valleys. Hence
\[
|V(z)|
=j+\sum_{i=1}^{\ell+1}
\bigl(\min\{u_i,d_i\}-1\bigr)
=\mu(w).
\]
Thus we have
\[
\deg\bigl(V(w)\bigr)\geq\mu(w).
\]
Combining the two inequalities gives
\[
\deg\bigl(V(w)\bigr)=\mu(w).
\]
The claimed description of $f_{ij}$ now follows from the definition of the
$f$-triangle.
\end{proof}

Now we explore the row and column sums of the $f$-triangle.

\begin{corollary}
Let $(f_{ij})_{0\leq j\leq i\leq n-1}$ be the $f$-triangle of
$\Delta_{\Dyck_n}$.

\begin{enumerate}
\item
[(i)]
Then the column sum $f_{j-1} = \displaystyle{\sum_{i \geq 0}} f_{ij}$ is the Narayana number $N_{n,j+1}$.
        In particular, the $f$-vector
        $(f_{-1},\ldots, f_{n-2})$
        of $\Delta_{\Dyck_n}$ is $(N_{n,1},\ldots, N_{n,n})$.

\item
[(ii)]
For $0\leq i\leq n-1$, the sum of the entries in the $i$th row 
satisfies
\[
\sum_{j=0}^{i}f_{ij}
=
\Big|\,\bigl\{\,w\in\Dyck_n\,:\,\mu(w)=i\,\bigr\}\,\Big|.
\] 

\end{enumerate}
\end{corollary}

\begin{proof}
The faces of $\Delta_{\Dyck_n}$ of cardinality $j$ are in
bijection with the Dyck paths of semilength $n$ having exactly
$j$ valleys. Equivalently, these paths have $j+1$ peaks. Their
number is the Narayana number
\[
N_{n,j+1}
=
\frac{1}{n}
\binom{n}{j+1}\binom{n}{j};
\]
see \cite{Deu98}. This proves~(i).

Part~(ii) follows immediately from
\ref{thm:typeA-f}.

\end{proof}

Note that the number $\sum_{i \geq 0} f_{ii}$ counts the maximal faces of 
$\Delta_{\Dyck_n}$. The corresponding sequence of numbers can be found in 
the Online Encyclopedia of Integer Sequences \cite{OEIS} as A086581 (or A082582).
This follows from \ref{lem:max}
and the description of A086581.

The number $f_{ii}$ of Dyck paths in $\Dyck^{\max}_n$ with $i$ valleys seems to coincide with 
A271942 in \cite{OEIS}. In that case, the description of the sequence does not easily
match the definition of $f_{ii}$.

\begin{conjecture}
    The number $f_{ii}$ of Dyck paths in $\Dyck^{\max}_n$ with $i$ valleys coincides with 
    A271942 in \cite{OEIS}.
\end{conjecture}

Now we turn to the study of the $h$-triangle. 
Let $ w \in \Dyck_{n}^{\max} $ and let $\prec$ be the shelling order from 
the proof of \ref{thm:shellable}. If we translate the definition of
the restriction
of a maximal face $w$ from \ref{prop:shelh} to this situation, we get:
\[
R(w) = \big\{\,p \in V(w)\,:\,  V(w)\setminus\{p\} \subseteq V(w') \text{ for some } w' \in \Dyck_{n}^{\max} \text{ with } w' \prec w\,\big\}.
\]
By \ref{prop:shelh}, we have
$$
h_{i\,j} = \Big|\,\big\{\,w \in \Dyck_{n}^{\max} \,:\, |V(w)| = i \text{ and } |R(w)| = j\,\big\}\,\Big|.
$$
For the purpose of defining the concept of \textit{successive valleys}, we regard the initial point $(0,0)$ 
as a virtual valley. More precisely, a valley $(x,y)\in V(w)$ is called \emph{successive} if \[ (x-2,y)\in V(w)\cup\{(0,0)\}. \] Thus, whenever $(2,0)$ is a valley of $w$, then it is
regarded as successive. We denote by
\[
\suc(w)
=
\Big|\,\left\{
(x,y)\in V(w)
\,:\,
(x-2,y)\in V(w)\cup\{(0,0)\}
\right\}\,\Big|.
\] 
 the number of successive valleys in $w$.

We can now interpret the $h$-triangle in terms of statistics on
Dyck paths.

\begin{theorem}\label{thm:typeA-h}
For $0\leq j\leq i\leq n-1$,
\[
h_{i\,j} = \Big|\,\big\{\,w \in \Dyck_{n}^{\max}\,:\, |V(w)| = i \text{ and } \suc(w) = i - j\,\big\}\,\Big|.\]
\end{theorem}

\begin{proof}
Let $w\in\Dyck_n^{\max}$ have $i$ valleys, of which $i-j$ are
successive. Partition the valleys of $w$ into maximal blocks of
consecutive valleys.

If $(2,0)\notin V(w)$, each block with $r$ valleys contributes
$r-1$ successive valleys. Hence, the number of blocks is
$i-(i-j)=j.$
If $(2,0)\in V(w)$, the block containing $(2,0)$ contributes one
additional successive valley because $(0,0)$ is regarded as an
auxiliary valley. Hence, in this case, the total number of blocks
is $j+1$.

We claim that $R(w)$ consists precisely of the first valleys of the blocks not containing $(2,0)$.

We first show that no valley other than the first valley of a
block belongs to $R(w)$. Let
\[
p=(x,y)
\qquad\text{and}\qquad
q=(x+2,y)
\]
be two consecutive valleys in the same block. Suppose, for a
contradiction, that $q\in R(w)$. Then there exists
$w'\in\Dyck_n^{\max}$ such that
\[w'\prec w
\qquad\text{and}\qquad
V(w)\setminus\{q\}\subseteq V(w').
\]
In particular, every valley of $w$ preceding $q$, including $p$,
also belongs to $V(w')$.
Since $w'\prec w$, the first valley at which the ordered valley sequences of $w'$ and $w$ differ must occur before $q$. Such a new valley cannot occur between $p$ and $q$, since consecutive valleys are separated by the shortest possible subpath $UD$. 
Next we argue, that it also cannot occur between two earlier consecutive
preserved valleys of $w$. Indeed, replacing the corresponding
subpath of $w$ by the subpath of $w'$ between the same endpoints
would produce a Dyck path whose valley set properly contains
$V(w)$. This would contradict the fact that $V(w)$ is a maximal face.
Consequently, no new valley of $w'$ can occur before $q$.
Therefore, the first difference between the valley sequences of
$w$ and $w'$ occurs at $q$. Since $q\notin V(w')$, this implies
$w\prec w',$ a contradiction. Thus, $q\notin R(w).$

The valley $(2,0)$ does not belong to $R(w),$ since it is the
earliest possible valley of a Dyck path. Indeed, if
\[
V(w)\setminus\{(2,0)\}\subseteq V(w')
\]
for some other maximal face $V(w')$, then the first valley at which
$w'$ differs from $w$ must occur after $(2,0)$, and hence
$w\prec w'$. Therefore, $(2,0)\notin R(w).$

It remains to show that the first valley of every block not containing $(2,0)$ belongs to $R(w)$. Let $p = (x,y)$ be such a valley, represented by the factor $DU$ in $w = w_1DUw_2$. We distinguish two cases.

\medskip
\noindent\textsf{Case 1: $(x+2,y) \not\in V(w)$.}
First, suppose $w_1 = w_1'D$ and  
$w=w_1'DDUw_2.$
Define
$w'=w_1'DUDw_2.$
It replaces $p$ by an
earlier valley and preserves all other valleys. Thus $w' \prec w$ and
$V(w) \setminus \{p\} \subseteq V(w')$. We conclude $p \in R(w)$.

Next, suppose $w_1 = w_1'U$ and $w=w_1'UDUw_2.$ 
Since
$p$ is the first valley of its block and $p\neq(2,0)$, the
preceding step must also be a $U$-step. Hence,
$w=w_1''UUDUw_2.$

If $w_2$ does not begin with $DD$, define
$w'=w_1''UDUUw_2$ which then is $UUDD$ avoiding.
The only prefix of $w'$ representing a path ending below the 
path for the corresponding prefix of
$w$ is $w_1''UD$. The height at this point is the same as the height
after $w_1''$ and hence $\geq 0$.
Thus, $w'$ is a Dyck path.

Now consider the case $w_2=DDw_2'$. Then we set 
$w' = w_1''UDUDUDw_2'$. Since $w$ is $UUDD$-avoiding it follows
that $w'$ is $UUDD$-avoiding. By an argument similar to the ones above
we also conclude that $w'$ is a Dyck path. 

We have shown that in each subcase, 
$w'\in\Dyck_n^{\max}.$
Moreover, the first new valley of $w'$ precedes $p$ and every valley
of $w$ other than $p$ is preserved. Therefore, we have
\[
w'\prec w,
\qquad
V(w')\cap V(w)=V(w)\setminus\{p\}.
\]
It follows that $p\in R(w)$.

\medskip

\noindent\textsf{Case 2: $(x+2,y) \in V(w)$.}
In that case $w_2 = DUw_2'$ and $w = w_1DUDUw_2'$. 
 
First, suppose $w_1 = w_1'D$ and $w=w_1'DDUDUw_2'$.
Define
$w'=w_1'DUDDUw_2'.$
This replacement moves a $U$ to the left and hence preserves
the Dyck path condition.

Next, suppose that $w_1 = w_1'U$ and $w=w_1'UDUDUw_2'$.
Since $p$ is the first valley of its block and is different from 
$(2,0)$ we must have $w_1' = w_1''U$ with a possibly empty 
$w_1''$. Hence we have $w=w_1''UUDUDUw_2'$.
Define
$w'=w_1''UDUUDUw_2'$.
The only prefix of $w'$ representing a path ending below the
path represented by the corresponding prefix of $w$ is 
$w_1''UD$. Since the height of the endpoint of this path is the
same as the height at the end of $w_1''$ it follows that $w'$ is
a Dyck path. 

In both subcases, the $w'$ is easily seen to be $UUDD$-avoiding,
satisfies
$w'\prec w,$
and has
$V(w')\cap V(w)=V(w)\setminus\{p\}.$
Again we conclude $p\in R(w)$.

This shows that $R(w)$ consists precisely of the first valleys
of the blocks not containing $(2,0)$. If $(2,0)\notin V(w)$,
there are $j$ such blocks. If $(2,0)\in V(w)$, there are $j+1$
blocks in total, but the block containing $(2,0)$ contributes
nothing to $R(w)$. Hence, in both cases,
$|R(w)|=j.$
The result now follows from the interpretation of
$h_{ij}$ for the shellable complex $\Delta_{\Dyck_n}$.
\end{proof}

The next corollary follows from \ref{thm:typeA-h} and \cite[Theorem 4.1]{BjoWac96}.

\begin{corollary}
    The complex $\Delta_{\Dyck_n}$ has the homotopy type of a wedge of spheres. 
    The number of spheres of dimension $i-1$ in the wedge is
    
    $$h_{i\,i}=\Big|\,\Big\{\,w\in \Dyck_n^{\max}\,:\,\genfrac{}{}{0pt}{}{|V(w)|=i \text{ and}}{\suc(w) = 0} \,\Big\}\,\Big|.$$
    In particular, $$\rk_\ZZ(\,\tilde{H}_{i-1}(\Delta_{\Dyck_n},\ZZ)\,)=
    \Big|\,\Big\{\,w\in \Dyck_n^{\max}\,:\,\genfrac{}{}{0pt}{}{|V(w)|=i \text{ and}}{\suc(w) = 0} \,\Big\}\,\Big|.$$
\end{corollary}

The following is an immediate consequence of \ref{thm:typeA-f}, \ref{thm:typeA-h} together with
the transformation formulas \cite[Definition 3.1,(3.4)]{BjoWac96}.

\begin{corollary} \label{cor:fh-ident}
\begin{align*}
     \Big|\,\Big\{\,&w \in \Dyck_{n}^{\max}\, :\, |V(w)| = i \text{ and } \suc(w) = i - j \,\Big\}\,\Big|  = \\ & \sum_{k=0}^ j (-1)^ {j-k} \binom{i-k}{j-k}  \Big|\,\Big\{ w \in \Dyck_n~:~|V(w)| = k \text{ and}\ \mu(w)=i \Big\}\,\Big|\\
    \Big|\,\Big\{\,& w \in \Dyck_n~:\, |V(w)| = j  \text{ and } \mu(w)=i\, \Big\}\,\Big|  = \\
     & \sum_{k=0}^j \binom{i-k}{i-j}  \Big|\,\Big\{\,w \in \Dyck_{n}^{\max}\, :\,  |V(w)| = i \text{ and } \suc(w) = i - k \,\Big\}\,\Big|.
    \end{align*}
\end{corollary}

It would be interesting to find a more direct combinatorial proof of the preceding corollary.

\begin{question}
    Is there a direct proof of \ref{cor:fh-ident} based on the combinatorics of Dyck paths?
\end{question}

Even though the $f$- and $h$-triangles of the complex 
$\Delta_{n}^{(a,b)}$ of rational Dyck paths are non-negative (for the $h$-triangle this is
\ref{cor:wedge2}), we were not able to extend \ref{thm:typeA-f} and \ref{thm:typeA-h} to rational Dyck paths. Of particular interest is the case $(a,b) = (1,b)$ in which the
number of rational Dyck paths is counted by the Fuss-Catalan numbers.
We refer the reader to \cite{Sta15} for background
on Fuss-Catalan numbers and ask the following question.

\begin{question}
    Which statistics on the rational Dyck paths $\Dyck_n^{(a,b)}$ are underlying the 
    $f$- and $h$-triangles of $\Delta_n^ {(a,b)}$?
\end{question}

In the Appendix we have listed a few examples of the $f$- and $h$-triangles for
$\Delta_n^{(1,b)}$ for $n=3,4,5,6$ and $b=2,3$ as inspiration for studying the question.

\section{Type \texorpdfstring{$B_n$}~: symmetric Dyck paths} \label{sec:b}

A Dyck path
$w=w_1w_2\cdots w_{2m}$
of semilength $m$ is called \textit{symmetric} if
\[
w_{2m+1-i}=\overline{w_i}
\qquad\text{for all }1\leq i\leq 2m,
\]
where $\overline{U}=D$ and $\overline{D}=U$. 
Geometrically, such a path is symmetric with respect to
the vertical line $x=m$. For a word $w = w_1\cdots w_k$ with letters
$w_1,\ldots, w_k \in \{D,U\}$ we write 
$\overline{w}$ for the word $\overline{w_1}\cdots \overline{w_k}$ and 
$w^{\rev}$ for the word $w_k\cdots w_1$.
In particular, if $w$ is a symmetric Dyck path, then 
$w = \overline{w}^{\rev} = \overline{w^{\rev}}$.

The following example illustrates the well-known correspondence between antichains in $\Phi_+(B_n)$ and symmetric Dyck paths of semilength $2n$.  A type $B_n$ Dyck path is a lattice path consisting of $2n$ steps that starts at $(0,0)$, never goes below the horizontal axis, and ends at $ (2n,2y) $ for some $0\leq y\leq n$.  

Adjoining to a type $B_n$ Dyck path its reflection across the vertical line $x=2n$ produces a classical Dyck path from $(0,0)$ to $(4n,0)$ that is symmetric. Conversely, the first half of every such symmetric Dyck path 
of semilength $2n$ is a type $B_n$ Dyck path. Thus, type $B_n$ Dyck paths can be naturally identified with symmetric classical Dyck paths of length $4n$, or equivalently, of semilength $2n$.

We denote by
$\SDyck_{2n}\subseteq\Dyck_{2n}$
the set of symmetric Dyck paths of semilength $2n$. For
$w\in\SDyck_{2n}$, let
\[
V^{\half}(w)
=
\bigl\{\,(x,y)\in V(w)\,:\,x\leq 2n\,\bigr\}
\]
be the set of valleys lying in the first half of $w$, including
a possible valley on the axis of symmetry. We define
\[
\Delta_{\SDyck_{2n}}
=
\bigl\{\,
V^{\half}(w)
\,:\,
w\in\SDyck_{2n}
\,\bigr\}.
\]
Under the correspondence described above,
$\Delta_{\SDyck_{2n}}$ is naturally isomorphic to the antichain
complex $\Delta_{B_n}$ and is therefore a simplicial complex. The next
example illustrates this correspondence.

\begin{example} \label{ex:bn}
   Let $\fD = B_n$ with root system 
   $$\Phi(B_n) = \big\{\, \pm (e_i-e_j), \pm(e_i+e_j)~:~ 1 \leq i < j \leq n
   \,\big\}
   \cup \big\{~\pm e_i ~:~1 \leq i \leq n\,\big\}.$$
   We choose the 
   simple roots $e_1 - e_2, e_2 - e_3,
   \ldots, e_{n-1} - e_n, e_n$.
   Then the positive roots
   are $\Phi_+(\fD) = \{ e_i \pm e_j \,:\,
   1 \leq i < j \leq n\} \cup \{ \,e_i~:~1 \leq i \leq n\}$.
   See \ref{fig:b3} for the
   Hasse diagram of the root poset for $B_3$ and
   the corresponding half of the symmetric Dyck path of semilength $6 = 2\cdot 3$ 
   associated with the antichain $\{e_2-e_3\}$.
   
\begin{figure}
   \begin{tikzpicture}
 
\node (a1)  at (-2,0)  {};
 \fill (a1) circle (2.5pt) node[right] {$~~e_1-e_2$};
 
\node (a2)  at ( 0,0)  {};
 \fill (a2) circle (2.5pt) node[right] {$~~e_2-e_3$};
 
\node (a3)  at ( 2,0)  {};
\fill (a3) circle (2.5pt) node[right] {$~~e_3$};

\node (a12) at (-1,1) {};
\fill (a12) circle (2.5pt) node[right] {$~~e_1-e_3$};

\node (a23) at ( 1,1) {};
\fill (a23) circle (2.5pt) node[right] {$~~e_2$};

\node (a123) at (0,2) {};
\fill (a123) circle (2.5pt) node[right] {$~~e_1$};

\node (a2_2a3) at (2,2) {};
\fill (a2_2a3) circle (2.5pt) node[right] {$~~e_2+e_3$};

\node (a12_2a3) at (1,3) {};
\fill (a12_2a3) circle (2.5pt) node[right] {$~~e_1+e_2$};

\node (top) at (2,4) {};
\fill (top) circle (2.5pt) node[right] {$~~e_1+e_3$};

\draw (a1) -- (a12);
\draw (a2) -- (a12);

\draw (a2) -- (a23);
\draw (a3) -- (a23);

\draw (a12) -- (a123);
\draw (a23) -- (a123);

\draw (a23) -- (a2_2a3);

\draw (a123) -- (a12_2a3);
\draw (a2_2a3) -- (a12_2a3);

\draw (a12_2a3) -- (top);

 \node (cp0) at (4,0) {};
 \fill[blue] (cp0) circle (2.5pt) ;
 \node (cp1) at (5,1) {};
 \fill[blue] (cp1) circle (2.5pt) ;
 \node (cp2) at (6,2) {};
 \fill[blue] (cp2) circle (2.5pt) ;
 
\node (ca1)  at (6,0)  {};
 \fill (ca1) circle (2.5pt) ;
 
\node (ca2)  at ( 8,0)  {};
 \fill[red] (ca2) circle (2.5pt) node[right] {$~~e_2-e_3$};
 
\node (ca3)  at ( 10,0)  {};
\fill (ca3) circle (2.5pt) ;

\node (ca12) at (7,1) {};
\fill (ca12) circle (2.5pt) ;

\node (ca23) at ( 9,1) {};
\fill (ca23) circle (2.5pt);

\node (ca123) at (8,2) {};
\fill (ca123) circle (2.5pt);

\node (ca2_2a3) at (10,2) {};
\fill (ca2_2a3) circle (2.5pt);

\node (ca12_2a3) at (9,3) {};
\fill (ca12_2a3) circle (2.5pt);

\node (ctop) at (10,4) {};
\fill (ctop) circle (2.5pt);

\draw (ca1) -- (ca12);
\draw[blue, very thick] (ca2) -- (ca12);
\draw[blue,very thick] (cp0) -- (cp1);
\draw[blue,very thick] (cp1) -- (cp2);
\draw[blue,very thick] (cp2) -- (ca12);
\draw[blue,very thick] (ca2) -- (ca23);
\draw (ca3) -- (ca23);

\draw (ca12) -- (ca123);
\draw (ca23) -- (ca123);

\draw[blue,very thick] (ca23) -- (ca2_2a3);

\draw (ca123) -- (ca12_2a3);
\draw (ca2_2a3) -- (ca12_2a3);

\draw (ca12_2a3) -- (ctop);
\end{tikzpicture}

\caption{Root Poset of $B_3$ and Type $B_3$ Dyck-Path Corresponding to the Antichain $\{e_2-e_3\}$} \label{fig:b3}
\end{figure}
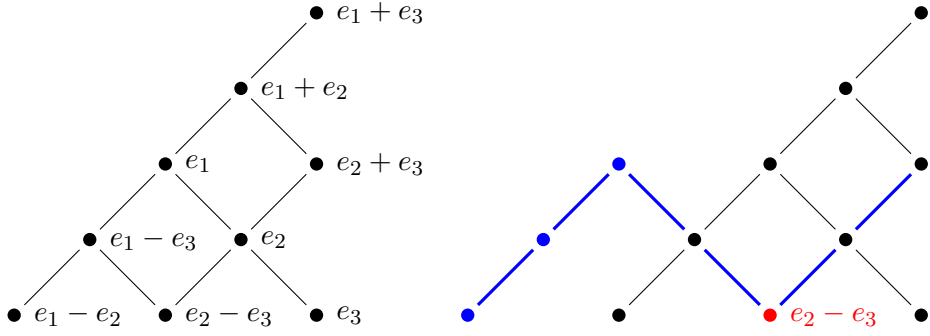
   \end{example}

We next characterize the symmetric Dyck paths corresponding to
the maximal faces of $\Delta_{\SDyck_{2n}}$. We write
\[
\SDyck_{2n}^{\max}
=
\bigl\{\,
w\in\SDyck_{2n}
\,:\,
V^{\half}(w)\text{ is a maximal face of }
\Delta_{\SDyck_{2n}}
\,\bigr\}.
\]

\begin{lemma} \label{lem:SDyck-max}

Let $w\in\SDyck_{2n}$ be a symmetric Dyck path of length $4n$.
Then $V^{\half}(w)$ is a maximal face of
$\Delta_{\SDyck_{2n}}$ if and only if $w$ is $UUDD$-avoiding.
\end{lemma}

\begin{proof}

Suppose first that $w$ is $UUDD$-avoiding. Then by
\ref{lem:max} the $V(w)$ is a maximal face of the
Dyck complex $\Delta_{\Dyck_{2n}}$.

Assume, for contradiction, that $V^{\half}(w)$ is not a maximal face
of $\Delta_{\SDyck_{2n}}$. Then there exists
$z\in\SDyck_{2n}$ such that
$V^{\half}(w)\subsetneq V^{\half}(z).$
Since both $w$ and $z$ are symmetric, every valley lying strictly
in the first half has a corresponding reflected valley in the
second half, while the valley on the axis of symmetry, if it exists, 
is fixed by the reflection. Therefore,
$V(w)\subsetneq V(z).$
This contradicts the maximality of $V(w)$ in
$\Delta_{\Dyck_{2n}}$. Hence $V^{\half}(w)$ is a maximal face of
$\Delta_{\SDyck_{2n}}$.

Conversely, suppose that $w$ is not $UUDD$-avoiding. We show
that $V^{\half}(w)$ is not maximal by modifying $w$ preserving the
property of being symmetric. 

First, we consider the case when
$w = w_1UUDDw_2UUDDw_3$, where
$w_1 = \overline{w_3}^\rev$ and $w_2 = \overline{w_2}^{\rev}$.
Then $z = w_1UDUDw_2UDUDw_3$ is a Dyck path since the height in the
first factor $UDUD$ never goes below the height at the end of $w_1$ 
and the height of the second factor $UDUD$ never goes below the height at 
the end of $w_1UDUDw_2$. The symmetry of $z$ is obvious and
$z \in \SDyck_{2n}$. 

Moreover, each factor $UDUD$ creates one new valley and does not
destroy any existing valley. The two new valleys are reflections
of one another, and exactly one of them lies in the first half.
Consequently,
$V^{\half}(w)\subsetneq V^{\half}(z).$

Two subwords $UUDD$ cannot overlap, since none of the proper prefixes of $UUDD$ is a suffix. Hence the only case left to consider is
when $w = w_1UUDDw_2$ with $w_1 = \overline{w_2}^{\rev}$. 
Define $z = w_1UDUDw_2$.

Clearly, $z$ is again symmetric. As above, the
replacement preserves the Dyck path condition and creates a new
valley on the axis of symmetry. Therefore,
$V^{\half}(w)\subsetneq V^{\half}(z).$

In either case, $V^{\half}(w)$ is properly contained in another
face of $\Delta_{\SDyck_{2n}}$, and hence it is not a maximal face. This
completes the proof.
\end{proof}

Let $\prec$ be the linear order on $\Dyck_{2n}^{\max}$ introduced before the
proof of \ref{thm:shellable} in \ref{sec:a}. Since by \ref{lem:SDyck-max} we have
$\SDyck_{2n}^{\max} \subseteq \Dyck_{2n}^{\max}$ the order $\prec$ defines
also a linear order on $\SDyck_{2n}^ {\max}$. Indeed, since the 
valleys of $w \in \SDyck_{2n}$ are determined by $V^{\half}(w)$, 
we can decide $v \prec w$ by comparing $V^{\half}(v)$ and $V^ {\half}(w)$.

\begin{theorem}\label{thm:SDyck-shellable} 
The simplicial complex $\Delta_{\SDyck_{2n}}$ is shellable. \end{theorem} 
\begin{proof} 
We proceed parallel in with the proof of \ref{thm:shellable}.
  Let $v \prec w$ be Dyck words corresponding to maximal faces
    \[V^{\half}(v) = \big\{\,(x_1,y_1),\dots,(x_\ell,y_\ell)\,\big\} \text{ and } V^ {\half}(w) = \big\{\,(x_1',y_1'),\dots,(x_{\ell'}',y_{\ell'}')\,\big\} \] of $\Delta_{\SDyck_n}$. As usual we assume
    $x_1 < \cdots < x_\ell$ and $x_1' < \cdots < x_{\ell'}'$.
    Let $(x_k,y_k) \in V^{\half}(w)$ be the lexicographically least valley in 
    $V^ {\half}(w) \setminus V^{\half}(v)$. 
    
Let $w = w_1DUw_2$ where $|w_1| = x_k-1$. As in the proof of \ref{thm:shellable} we consider the case when $w_1$ ends with $D$ and the case when it ends with $U$ separately.

\medskip

    \noindent {\sf Case:} $w_1 = w_1'D$

    In this case the $z \prec w$ constructed in the proof of \ref{thm:shellable} is $z = w_1'DUDw_2$. If $|w_1D| = x_k = 2n$ then this
    $z$ is already symmetric and if $|w_1D| <2n$ then the change can be
    easily mirrored in $w_2$ to make $z$ symmetric.

\medskip

    \noindent {\sf Case:} $w_1 = w_1'U$

    The same arguments as in the proof of \ref{thm:shellable} 
    show that $w_1$ does not factor as $w_1 = w_1'' DU$ or equivalently there is no valley immediatly preceding $(x_k,y_k)$. 
    Thus we have $w_1 = w_1''UU$. 
    
    We continue to follow the proof of \ref{thm:shellable} and distinguish
    subcases.

    \medskip

    $\rightarrow$ $w_2 = Uw_2'$

    In this case $|w_1DUU| = x_k+2 \leq 2n$ and 
    $z = w_1''UDUUUw_2'$ from the proof of \ref{thm:shellable} 
    can be easily made symmetric by mirroring the
    changes in $w_2'$ to make $z$ a symmetric Dyck path. The 
    arguments from \ref{thm:shellable} then complete the case.

    \medskip

    $\rightarrow$ $w_2 = DUw_2''$

    In that case $w = w_1''UUDUDUw_2'$ and the proof of 
    \ref{thm:shellable} considers the path $z = w_1''UDUUDUw_2''$.
    Since $w$ cannot be $w_1''UUD\overline{w_1''UUD}^{\rev}$
    it follows that $|w_1''UUD| < 2n$ and this change can be easily mirrored in $w_2''$
    to make $z$ a symmetric Dyck word. 

    \medskip

    $\rightarrow$ $w_2 = DDw_2''$
    
    In that case $w = w_1''UUDUDDw_2'$ and the proof of 
    \ref{thm:shellable} considers the path $z=w_1''UDUDUDw_2''$. 
    As in the preceding subcase it follows that  $|w_1''UUD| \leq 2n$. This and the structure of the replacement 
    imply that $z$ can be easily transformed into a symmetric Dyck path by 
    modifying $w_2''$. 

In all cases, $z$ is a $UUDD$-avoiding symmetric Dyck path
and $z \prec w$. 
Direct inspection shows that $(x_k,y_k) \not\in V^ {\half}(z)$ 
and that all
other valleys from $w$ are inherited by $z$. 
Thus 
\[
V^{\half}(z)\cap V^{\half}(w)
 =
V^{\half}(w)\setminus\{ (x_k,y_k)\}.
\]

It follows that $\prec$ is a shelling order and $\Delta_{\SDyck_{2n}}$ is
a shellable simplicial complex.
\end{proof}

Next we aim to find combinatorial interpretations for the $f$- and
$h$-triangle of $\Delta_{\SDyck_{2n}}$. We start by defining a
type $B$-analog of the function $\mu$. 

For a point $(x,y)$ reached by a path $w \in \{D,U\}^ *$ the number of 
letters $D$ and the number of letters $U$ in $w$ can be computed as 
\[
d\big(\,(x,y)\,\big)=\frac{x-y}{2}
\quad \text{ and } \quad
u\big(\,(x,y)\,\big)=\frac{x+y}{2}.
\]

Let $w\in\SDyck_{2n}$, and write
\[
V^{\half}(w)=\big\{\,(x_1,y_1) , \ldots , (x_\ell,y_\ell)\,\big\}
\]
where $x_1<\cdots<x_\ell$. For $k = 1,\ldots, \ell$ set
\[
d_k=d\big(\,(x_k,y_k)\,\big)= \frac{x_k-y_k}{2} 
\quad \text{ and } \quad
u_k=u\big(\,(x_k,y_k)\,\big) = \frac{x_k+y_k}{2}.
\]
If $\ell \geq1$ or equivalently $V^ {\half}(w) \neq \varnothing$  
define
\[
\mu_B(w)
=
d_1+
\sum_{k=1}^{\ell-1}
\min\{d_{k+1}-d_k,u_{r+k}-u_k\}
+
\left\lfloor
\frac{2n-d_\ell-u_\ell}{2}
\right\rfloor.
\tag{5.1}
\]
If $V^{\half}(w)=\varnothing$, we set
$\mu_B(w)=n.$

In terms of the coordinates $(x_k,y_k)$,
\[
\mu_B(w)
=
\frac{x_1-y_1}{2}
+
\sum_{k=1}^{\ell-1}
\frac{x_{k+1}-x_k-|y_{k+1}-y_k|}{2}
+
\left\lfloor\frac{2n-x_\ell}{2}\right\rfloor.
\tag{5.2}
\]

Now we are in a position to describe the $f$-triangle of $\Delta_{\SDyck_{2n}}$.

\begin{theorem}\label{thm:typeB-f}
Let
$(f^B_{ij})_{0\leq j\leq i\leq n}$
be the $f$-triangle of $\Delta_{\SDyck_{2n}}$. Then we have 
\[
f^B_{ij}
=
\Big|\,\left\{
w\in\SDyck_{4n}
\,:\,
|V^{\half}(w)|=j
\text{ and }
\mu_B(w)=i
\right\}\,\Big|.
\]
\end{theorem}

\begin{proof}
Consider the set
\[
P_n=\left\{(x,y)\in\mathbb{N}^2\,:\,1\leq x\leq y,\quad x+y\leq2n
\right\},
\]
equipped with the strict product order
\[
(x,y)\sqsubset (x',y')
\quad\Longleftrightarrow\quad
d<d'
\ \text{and}\
u<u'.
\]

We identify a valley $(x,y)$ in the first half of a symmetric Dyck path $w \in \SDyck_{2n}$
with 
\[
\big(\,d(\,(x,y)\,),u(\,(x,y)\,)\,\big)
=
\left(\frac{x-y}{2},\frac{x+y}{2}\right)
\in P_n.
\]
As usual for $w \in \SDyck_{2n}$ we have
$V^{\half}(w) = \big\{\, (x_1,y_1),\ldots,(x_\ell,y_\ell)\,\big\}$ with $x_1 < \cdots < x_\ell$.
For $1 \leq k < k' \leq \ell$ and valleys $(x_k,y_k)$ and $(x_{k'},y_{k'})$ both
the number of $D$-steps and the number of $U$-steps preceding the respective valley increase.
Consequently,
$(d_k,u_k)\sqsubset (d_{k'},u_{k'}).$
It follows that the set $V^{\half}(w)$ can be identified
with a chain $C_w$ in $P_n$. Conversely, every chain $C$ in $P_n$
determines $w \in {\SDyck_{2n}}$ with $C = C_w$. Under this
identification, the maximal faces correspond to maximal chains.

We now calculate the maximum size of a chain $C'$ in $P_n$ containing
the fixed chain
$C_w=\big\{\,(d_1,u_1) \sqsubset \cdots \sqsubset (d_\ell,u_\ell)\,\big\}.$

Since the first coordinates of a chain in $P_n$ are strictly positive and strictly increasing, the chain $C'$ can contain at most $d_1$ elements less than or equal to 
$(d_1,u_1)$ in the order on $P_n$. 

Let $1 \leq k \leq \ell-1$. Then the number of elements $(x,y) \in C' \setminus C$ with $(d_k,u_k)  \sqsubset (x,y) \sqsubset (d_{k+1},u_{k+1})$ is 
 at most
$\min\{d_{k+1}-d_k,u_{k+1}-u_k\}.$
Indeed, both coordinates must increase by at least one at every
step.

Finally, suppose that $C'$ contains $s$ elements strictly
larger than $(d_\ell,u_\ell)$ in the order on $P_n$. 
Let $(x,y)$ be the largest among those. 
Then $x \geq d_\ell+s$ and $y \geq u_\ell+s$. 
Since every element $(x,y)$ of $P_n$ satisfies $x+y\leq2n$, we obtain
$d_j+u_j+2s\leq2n.$
Therefore,
\[
s\leq
\left\lfloor
\frac{2n-d_\ell-u_\ell}{2}
\right\rfloor.
\]
Combining these bounds and comparing with the formula for $\mu_B(w)$ gives
$|C'|\leq\mu_B(w)$ and hence $\deg(V^ {\half}(w)) \leq \mu_B(w)$,
for every chain $C'$ containing $C$.

It remains to show that this bound is attained. We add to $C$ the
elements of the chain
\[
(1,u_1-d_1+1)
\sqsubset
(2,u_1-d_1+2)
\sqsubset\cdots\sqsubset
(d_1,u_1).
\]
For each $1\leq k<\ell$, let
\[
m_k=
\min\{d_{k+1}-d_k,u_{k+1}-u_k\}
\]
and add to $C$ the elements
$(d_k+j,u_k+j),$
$1\leq j<m_k,$
between $(d_k,u_k)$ and $(d_{k+1},u_{k+1})$. 
Finally, add to $C$ the elements
\[
(d_\ell+j,u_\ell+j),
\qquad
1\leq s\leq
\left\lfloor
\frac{2n-d_\ell-u_\ell}{2}
\right\rfloor.
\]
All elements added to $C$ belong to $P_n$ and together they form a
chain $C'$ containing $C$. By construction the cardinality 
of $C'$ is $\mu_B(w)$. Hence $\deg(V^ {\half}(w)) \geq \mu_B(w)$. 

Combining the two inequalities we obtain
$\deg(V^ {\half}(w)) = \mu_B(w)$.

The assertion now follows from the definition of
the $f$-triangle.
\end{proof}

Next we determine a description of the $h$-triangle of
$\Delta_{\SDyck_{2n}}$.
Again translating the definition of the restriction of a maximal face from \ref{prop:shelh} to the given setting we obtain
\[
R(w)
=
\left\{
p\in V^{\half}(w)
\,:\,
\genfrac{}{}{0pt}{}{V^{\half}(w)\setminus\{p\}
\subseteq V^{\half}(w')}{
\text{ for some }
w'\in\SDyck_{2n}^{\max}
\text{ with } w'\prec w}
\right\}.
\]
By \ref{prop:shelh}, we have
\[
h^B_{i\,j}
=
\Big|\,\left\{\,
w\in\SDyck_{2n}^{\max}
\,:\,
|V^{\half}(w)|=i \text{ and } |R(w)|=j\,
\right\}\,\Big|.
\]

As in the classical case, we regard $(0,0)$ as a virtual valley for the purpose of 
defining successive valleys. Note that it is not a vertex of
$\Delta_{\SDyck_{2n}}$. Define
\[
\suc_B(w)
=
\Big|\,\left\{
(x,y)\in V^{\half}(w)
\,:\,
(x-2,y)\in V^{\half}(w)\cup\{(0,0)\}
\right\}\,\Big|.
\]
Thus $\suc_B(w)$ is the number of valleys in the first half of $w$ which are
immediately preceded by another valley at the same height, with
$(0,0)$ included as a virtual predecessor.

\begin{theorem}\label{thm:typeB-h}
For $0\leq j\leq i\leq n$, we have
\[
h^B_{i,j}
=
\Big|\,\left\{
w\in\SDyck_{2n}^{\max}
\,:\,
|V^{\half}(w)|=i \text{ and }
\suc_B(w)=i-j
\right\}\,\Big|.
\]
\end{theorem}

\begin{proof}
As in the proof of
\ref{thm:typeA-h}, partition
$V^{\half}(w)\cup\{(0,0)\}$
into maximal blocks of successive valleys. We claim that $R(w)$
consists precisely of the first valley of each block not containing
the virtual valley $(0,0)$.

The argument showing that a non-initial valley of a block
does not belong to the restriction set is identical to the  type
$A$ argument and actually follows from it. 

Now let $p = (x,y) \in V^ {\half}(w) $ be the first valley of a block corresponding to the $DU$ factor in the factorization 
$w = w_1DUw_2$. 
We now distinguish the same cases as in the proof of \ref{thm:typeA-h}.

\medskip

\noindent {\sf Case 1:} $(x+2,y) \not \in V^{\half}(w)$.
If $w_1 = w_1'D$ then by a symmetric modification of $w_2$ the 
path corresponding to $w_1'DUDw_2$ can be made into a symmetric $w' \prec w$ with $V^{\half}(w)\setminus\{p\}
\subseteq V^{\half}(w')$. It follows that $p\in R(w)$. 
In case $x = 2n$, the symmetry prevents the modification of $w_2$ to start with $DD$ and hence we always have $w' \in \SDyck^{\max}_{2n}$. 

Next consider the case $w_1 = w_1'U$. Since $(x,y)$ is the first valley in its block, we have $w_1 = w_1''UU$.
If $w_2$ does not start with $DD$, then 
by a suitable replacement of $w_2$ in 
$w_1''UDUUw_2 \in \Dyck_{2n}^{\max}$ we obtain $w'
\in \SDyck_{2n}^{\max}$ with $w' \prec w$ and $V^{\half}(w)\setminus\{p\}
\subseteq V^{\half}(w')$. It follows that $p\in R(w)$.

Consider the case $w_2 = DDw_2'$, then by symmetry of $w$, we must have $x < 2n-3$. Then by a suitable replacement of $w_2''$ in $w_1''UDUDUDw_2'' \in \Dyck_{2n}^{\max}$ 
we obtain a symmetric $w' \in \SDyck_{2n}^{\max}$
with $w' \prec w$ and $V^{\half}(w)\setminus\{p\}
\subseteq V^{\half}(w')$. It follows that $p\in R(w)$.

\medskip

\noindent {\sf Case 2:} $(x+2,y) \in V^{\half}(w)$.
Equivalently we have $w_2 = DUw_2'$.
Again we start with the case $w_1 = w_1'D$. 
In that case, by checking possible symmetry axes of $w_1'DDUDUw_2'$, we find that $x = 2n$ is not possible. 
Thus $x < 2n$. 
If $x \neq 2n-1$, consider
$w_1'DUDDUw_2' \in \Dyck_{2n}^{\max}$. Then by a suitable replacement of $w_2'$ this can be made into
$w' \in \SDyck_{2n}^ {\max}$. 
In the case $x= 2n-1$, the symmetry of $w$ implies $w_2 = DUUw_2''$.
Then by a replacement of $w_2''$, the Dyck path 
$w_1'DUDUDUw_2' \in \Dyck_{2n}^{\max}$ can be made into
$w' \in \SDyck_{2n}^{\max}$ which then satisfies $w' \prec w$ and
$V^{\half}(w)\setminus\{p\}
\subseteq V^{\half}(w')$. And then $p\in R(w)$.

Next consider the case $w_1 = w_1'U$. Using that $p = (x,y)$ is the first valley in its block, which in particular implies $x \neq 2$, it follows that $w_1 = w_1''UU$ 
and $ w= w_1''UUDUDUw_2'$. Checking possible axes of symmetry we deduce that $x < 2n$. If $x < 2n-1$ then by a suitable replacement of $w_2'$ the word 
$w_1''UDUUDUw_2' \in \Dyck_{2n}^{\max}$ can be made into a
symmetric $w' \in \SDyck_{2n}^{\max}$ such that 
$w' \prec w$ and $V^{\half}(w)\setminus\{p\}
\subseteq V^{\half}(w')$. We obtain that $p\in R(w)$.
Consider the case $x = 2n-1$. Then by symmetry $w_2' = DDw_2''$.
In that case, by a suitable replacement of $w_2''$, the Dyck path $w' = w_1''UDUDUDUDw_2''$ is symmetric and easily seen
to be maximal. Thus $w' \in \SDyck_{2n}^{\max}$ with $w' \prec w$
and $V^{\half}(w)\setminus\{p\}
\subseteq V^{\half}(w')$. It follows that $p\in R(w)$.

\end{proof}

\section{Type \texorpdfstring{$D$}~} \label{sec:d}

\subsection*{\texorpdfstring{$\mathbf{D_3}$}~} 

 The simplicial complex $\Delta_{D_3}$ is illustrated in \ref{fig:d3}, where it is clearly seen to be non-pure shellable. The homology groups, together with the corresponding $f$- and $h$-triangles, are also displayed in \ref{fig:d3}.
 
 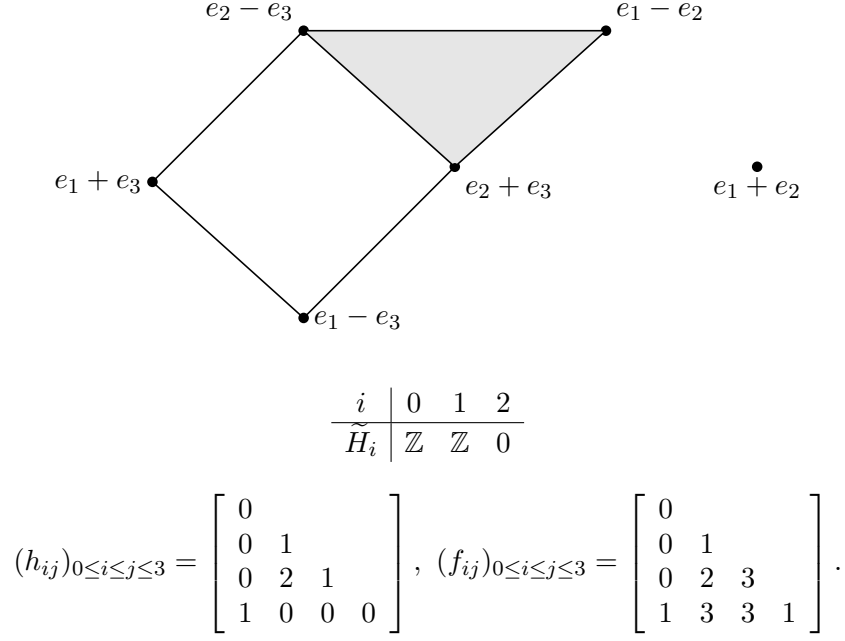
\begin{figure}[ht!]
 \begin{tikzpicture}[
    scale=1,
    vertex/.style={circle,fill,inner sep=1.4pt},
    edge/.style={line width=0.6pt},
    face/.style={fill=gray!20,draw=black,line width=0.6pt}
]

\node[vertex] (a) at (4,1.8) {};
\node[vertex] (b) at (0,1.8) {};
\node[vertex] (c) at (2,0) {};
\node[vertex] (d) at (-2,-0.2) {};
\node[vertex] (e) at (6,0) {};
\node[vertex] (f) at (0,-2) {};

\node[above right]  at (a) {$e_1-e_2$};
\node[above left] at (b) {$e_2-e_3$};
\node[below right] at (c) {$e_2+e_3$};
\node[left]  at (d) {$e_1+e_3$};
\node[below] at (e) {$e_1+e_2$};
\node[right] at (f) {$e_1-e_3$};

\filldraw[face] (a.center) -- (b.center) -- (c.center) -- cycle;

\draw[edge] (d.center) -- (b.center);
\draw[edge] (f.center) -- (c.center);
\draw[edge] (f.center) -- (d.center);

\foreach \v in {a,b,c,d,e,f}
    \fill (\v) circle (1.4pt);

\end{tikzpicture}

$$\begin{array}{c|ccc} i & 0 & 1 & 2 \\
\hline
\widetilde{H}_i & \ZZ & \ZZ & 0 
\end{array}$$

$$(h_{ij})_{0 \leq i \leq j \leq 3} = \left[ \begin{array}{cccc} 0 & & & \\ 0 & 1 && \\ 0 & 2 &1 & \\ 1 & 0 & 0 & 0 
\end{array} \right],\
(f_{ij})_{0 \leq i \leq j \leq 3} = \left[ \begin{array}{cccc} 0 &&& \\ 0 & 1 && \\ 0&2&3&\\ 1&3&3&1 
\end{array} \right].$$
\caption{The simplicial complex $\Delta_{D_3}$, its homology groups and $h$- and $f$-triangles} \label{fig:d3}
\end{figure}

\subsection*{\texorpdfstring{$\mathbf{D_4}$}~}

In order to give the reader a better feeling for why shellability fails for $D_n$ and $n \geq 4$, we provide an additional
separate proof for $n=4$.
Consider the root poset of $D_4$ in \ref{fig:d4}. The maximal antichain
$\{e_1-e_4,e_1+e_4,e_2+e_3\} \in \Delta_{D_4}$ has the property that it is the
unique maximal antichain containing any of 
its $2$-element subsets. 
If $\Delta_{D_4}$ were shellable, then there would be a shelling order starting with the unique $4$-element antichain,
it follows that $\Delta_{D_4}$ cannot be shellable. Nevertheless, \ref{fig:d4fh}
shows that the $h$-triangle still has non-negative entries, however, the homology groups cannot be read off from the diagonal of the $h$-triangle, which provides further evidence of non-shellability.

The non-shellability argument just presented for $D_4$ relies on the following well-known criterion.

\begin{lemma}\label{lem:nonshel} Let $\Delta$ be a simplicial complex of dimension $d$. 
Assume there
is a maximal face $F \in \Delta$ of dimension $d' < d$ such that for all 
$x \in F$,
 $F$ is the unique maximal face of $\Delta$ containing $F
\setminus \{x\}$. 
Then $\Delta$ is not shellable.
\end{lemma}

\begin{figure}[ht!]
\centering
\begin{tikzpicture}[
    scale=1,
    dot/.style={circle,fill,inner sep=2.5pt},
    edge/.style={line width=0.5pt}
]

\node[dot] (a1) at (-3,0) {};
\node[below left]  at (a1) {$e_1-e_2$};

\node[dot] (a2) at (-1,0) {};
\node[below left]  at (a2) {$e_2-e_3$};

\node[dot] (a3) at (1,0) {};
\node[below right] at (a3) {$e_3-e_4$};

\node[dot] (a4) at (3,0) {};
\node[below right] at (a4) {$e_3+e_4$};

\node[dot] (a12) at (-2,1.4) {};
\node[above left] at (a12) {$e_1-e_3$};

\node[dot] (a23) at (0,1.4) {};
\node[right] at (a23) {$~e_2-e_4$};

\node[dot] (a24) at (2,1.4) {};
\node[above right] at (a24) {$e_2+e_4$};

\node[dot] (a123) at (-2,2.8) {};
\node[above left] at (a123) {$e_1-e_4$};

\node[dot] (a124) at (0,2.8) {};
\node[above right] at (a124) {$e_1+e_4$};

\node[dot] (a234) at (2,2.8) {};
\node[above right] at (a234) {$e_2+e_3$};

\node[dot] (a1234) at (0,4.2) {};
\node[above right] at (a1234) {$e_1+e_3$};

\node[dot] (highest) at (0,5.6) {};
\node[above right] at (highest) {$e_1+e_2$};

\draw[edge] (a1) -- (a12);
\draw[edge] (a2) -- (a12);
\draw[edge] (a2) -- (a23);
\draw[edge] (a2) -- (a24);
\draw[edge] (a3) -- (a23);
\draw[edge] (a4) -- (a24);

\draw[edge] (a12) -- (a123);
\draw[edge] (a12) -- (a124);
\draw[edge] (a23) -- (a123);
\draw[edge] (a23) -- (a234);
\draw[edge] (a24) -- (a124);
\draw[edge] (a23) -- (a123);
\draw[edge] (a23) -- (a234);
\draw[edge] (a24) -- (a124);
\draw[edge] (a24) -- (a234);

\draw[edge] (a123) -- (a1234);
\draw[edge] (a124) -- (a1234);
\draw[edge] (a234) -- (a1234);

\draw[edge] (a1234) -- (highest);

\end{tikzpicture}
\caption{Root poset for $D_4$} \label{fig:d4}
\end{figure}
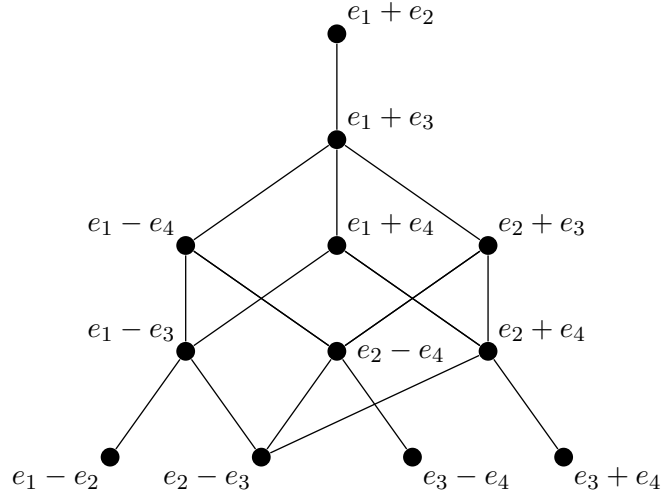

\begin{figure}[ht!]
{\small 
$$\begin{array}{c|cccc} i & 0 & 1 & 2 & 3\\
\hline
\widetilde{H}_i(\,\Delta_{D_4}\,) & \ZZ^2 & \ZZ^5 & \ZZ & 0 
\end{array}$$

$$(h_{ij})_{0 \leq j \leq i \leq 4} =  \left[ \begin{array}{ccccc}
  0 &&&& \\ 0 & 2 &&& \\ 0 & 0 & 6 && \\ 0 & 6 & 0 & 2 & \\ 1 & 0 & 0 &0 &0 
  \end{array} \right],\
(f_{ij})_{0 \leq j \leq i \leq 4} = \left[ \begin{array}{ccccc}
0 & & & & \\ 0 & 2 & & & \\ 0 & 8 & 6 & & \\ 0 & 6 & 12 & 8 & \\ 1&4&6&4&1 
\end{array} \right].
$$}
\caption{Homology groups and $h$- and $f$-triangle for $\Delta_{D_4}$}
\label{fig:d4fh}
\end{figure}

\subsection*{\texorpdfstring{$\mathbf{D_{\geq 5}}$}~}

We use the standard simple system of type $D_n$,
\[
\alpha_i=e_i-e_{i+1},\,\,\, 1\leq i\leq n-1,
\qquad
\alpha_n=e_{n-1}+e_n.
\]
The positive roots are $e_i-e_j$ and $e_i+e_j$, where
$1\leq i<j\leq n$, and their expansions in terms of the simple
roots are
$
e_i-e_j=\alpha_i+\cdots+\alpha_{j-1}
$
and
\begin{align} \label{6.1}
e_i+e_j & = 
\begin{cases}
\alpha_i+\cdots+\alpha_{j-1}
 +2(\alpha_j+\cdots+\alpha_{n-2})
 +\alpha_{n-1}+\alpha_n,
 & j<n,\\[1mm]
\alpha_i+\cdots+\alpha_{n-2}+\alpha_n,
 & j=n.
\end{cases}
\end{align} 

For case $n$ odd, we will need the following slight extension of
\ref{lem:nonshel}.

\begin{lemma}\label{lem:nonshel-family}
Let $\Delta$ be a simplicial complex of dimension $d$.
Assume there exists a non-empty collection $\cF$ of facets of $\Delta$ such that
\begin{itemize}
    \item if $F \in \cF$, then $\dim(F) < d$; and
    \item if for a facet $G$ of $\Delta$ there is $F \in \cF$ and $x \in F$ such that $F \setminus \{x\} \subseteq G$, then $G \in \cF$.
\end{itemize}
Then $\Delta$ is not shellable.
\end{lemma}

\begin{proof}
Suppose that $\Delta$ has a shelling order, and let $F$ be the
first facet in this order that belongs to $\cF$. The facet $F$ cannot be the first facet in the shelling, since
every facet in $\cF$ has dimension smaller than $d$.
By the shelling, there must be a facet $G$ preceding $F$ and an element
$x\in F$ such that
$
G\cap F=F\setminus\{x\}
$, and hence $F \setminus \{x\} \subseteq G$.
By the hypothesis on $\cF$, this implies 
$G\in\cF$, contradicting the choice of $F$.
\end{proof}

\begin{theorem}\label{thm:Dn-nonshellable}
For every $n\geq4$, the simplicial complex
$\Delta_{D_n}$ is not shellable.
\end{theorem}

\begin{proof}
By general facts about root posets and their antichains (see \cite{Arm09} or \cite{Rin18}), the root poset of type $D_n$ has width $n$, and its simple roots
form its unique $n$-element antichain. Hence
$
\dim\Delta_{D_n}=n-1.
$
We distinguish between even and odd $n$.

\medskip

\noindent\textbf{Case:} $n=2m$ is even.

Define
$
\beta_0=e_1-e_n
$
and
$
\beta_i=e_i+e_{n-i+1},
 1\leq i\leq m = \frac{n}{2},
$,
and set
$
F=\{\beta_0,\beta_1,\ldots,\beta_m\}.
$
Every root in $F$ is a sum of $n-1$ simple roots and hence has height $n-1$ in the root poset. It follows that $F$ is an antichain.
It has $m+1$ elements and hence dimension $m = \frac{n}{2} <n-1$.

We claim that every positive root $\gamma \notin F$ is comparable to at
least two elements of $F$.

First, let $\gamma=e_\ell-e_j$. Then
$
\gamma\leq\beta_0 = e_1-e_n.
$
If $j<n$, then also
$
\gamma\leq\beta_1=e_1+e_n.
$
If $j=n$, then $\ell \geq 2$ since $\gamma\neq\beta_0$, and
$
\gamma\leq\beta_2=e_2+e_{n-1}.
$

It remains to consider $\gamma=e_\ell+e_j$. If $j=n$, then by 
$\gamma\neq \beta_1$ we have $\ell\geq2$. In this case
$
\gamma\leq\beta_1
$
and
$
\gamma\leq\beta_2.
$
Suppose that $j<n$. A comparison of the coefficients in
\eqref{6.1} gives
\begin{align} \label{6.2}
\gamma\leq\beta_i
\quad & \Longleftrightarrow\quad
i\leq\ell
 \text{ and }
n-i+1\leq j,
\end{align}
and
\begin{align} \label{6.3}
\beta_i\leq\gamma
\quad & \Longleftrightarrow\quad
\ell\leq i
\text{ and }
j\leq n-i+1.
\end{align}

Put $r=n-j+1$. The equivalences~\eqref{6.2} and~\eqref{6.3} show that
$\gamma$ is comparable to every $\beta_i$ with $i$ lying
between $\ell$ and $r$, subject to $1\leq i\leq m$.
If $\ell=r$, then
$
\gamma=e_\ell+e_{n-\ell+1}=\beta_\ell\in F.
$
Otherwise, since $n=2m$, the interval between $\ell$ and $r$
contains at least two indices from $\{1,\ldots,m\}$. 

Hence every
positive root outside $F$ is comparable to at least two elements
of $F$.
It follows that, for every $\beta\in F$, no positive root outside
$F$ can be added to $F\setminus\{\beta\}$. Thus $F$ is the
unique facet containing $F\setminus\{\beta\}$. By
\ref{lem:nonshel}, $\Delta_{D_n}$ is not shellable.

\medskip

\noindent\textbf{Case:} $n=2m+1$ is odd.

We define
\[
B=
\big\{\,e_2-e_n\,\big\}
\cup
\big\{\,e_i+e_{n-i+2}\,:\,2\leq i\leq m+1\,\big\}.
\]
For every $j\in\{2,\ldots,n-1\}$, set
$
F_j=B\cup \big\{\,e_1-e_j\,\big\},
$
and let
$
\cF= \big\{\,F_j\,:\,2\leq j\leq n-1\,\big\}.
$

A comparison of the simple-root coefficients in \eqref{6.1} 
gives
\begin{align} \label{6.4}
\left\{
\,\gamma\in\Phi_+(D_n)\,:\,
\genfrac{}{}{0pt}{}{\gamma \text{ is incomparable}}{\text{with every element of } B}\,
\right\}
& =
\big\{\,e_1-e_j\,:\,2\leq j\leq n-1\,\big\}.
\end{align}

The roots on the right-hand side of \eqref{6.4} form a chain. It
follows that every $F_j$ is a maximal antichain.

Moreover, if $b\in B$, then for every $j \in \{2,\ldots, n-1\}$
another coefficientwise comparison
gives  that
\begin{align} \label{6.5}
\left\{
\,\gamma\in\Phi_+(D_n)\,:\,
\genfrac{}{}{0pt}{}{\gamma\text{ is incomparable with every}}{\text{element of }
(B\setminus\{b\})\cup\{e_1-e_j\}}
\,\right\}
=
\{b\}.
\end{align}

Consequently:

\begin{itemize}
\item every facet containing $B=F_j\setminus\{e_1-e_j\}$ is one
      of the facets $F_k$, where $2\leq k\leq n-1$;
\item for $b\in B$, the facet $F_j$ is the unique facet
      containing $F_j\setminus\{b\}$.
\end{itemize}

Thus, whenever a facet contains a codimension-one face of some
$F_j\in\cF$, it also belongs to $\cF$.
Since
$
|B|=m+1
$
and
$
|F_j|=m+2,
$
every facet in $\cF$ has dimension $m+1<2m=n-1 = \dim(\Delta_{D_n})$.
Therefore, \ref{lem:nonshel-family} implies that
$\Delta_{D_n}$ is not shellable.
\end{proof}

While the $h$-triangle for $D_5$ still has non-negative entries
(see \ref{fig:d5fh}), for $D_6$
this is no longer the case (see \ref{fig:d6fh}). Experimental data suggest that this failure occurs for all $n \geq 6$.

\begin{figure}[H]
{\small
$$\begin{array}{c|ccccc} i & 0 & 1 & 2 & 3 & 4   \\
\hline
\widetilde{H}_i(\,\Delta_{D_5}\,) & \ZZ^2  & \ZZ^{10} & \ZZ^8 & 0 & 0  
\end{array}$$

$$(h_{ij})_{0 \leq j \leq i \leq 5} = \left[ \begin{array}{cccccc}
0 & & & & &  \\ 0 & 2 & &&&\\ 0 & 1 & 10 &&& \\  0 & 4 & 11 & 9 &&  \\
0 & 8 & 6 & 0 & 1 & \\ 1 & 0& 0& 0& 0&0 
\end{array} \right],\
(f_{ij})_{0 \leq j \leq i \leq 5} = \left[ \begin{array}{cccccc} 
0 & & & & & \\
0 & 2 & & & & \\ 
0 & 1 & 11 & && \\ 
0 & 4 & 19 & 24 & & \\ 
0 & 8 & 30 & 36 & 15 &  \\ 
1 & 5 & 10 & 10 & 5 & 1 
\end{array} \right]$$}
\caption{Homology groups and $h$- and $f$-triangles for $\Delta_{D_5}$}
\label{fig:d5fh}
\end{figure}

\begin{figure}[H]
{\small
$$\begin{array}{c|cccccc} i & 0 & 1 & 2 & 3 & 4 & 5  \\
\hline
\widetilde{H}_i(\,\Delta_{D_6}\,) & \ZZ^2 & \ZZ^{14} & \ZZ^{26} & \ZZ^ 6 & 0 & 0 
\end{array}$$

$$(h_{ij})_{0 \leq j \leq i \leq 6} = \left[ \begin{array}{ccccccc}
0 & & & & & &  \\ 0 & 2 & & & & & \\
0 & 2 & 14 & & & & \\
0 & 2 & 26 & 26 & & &  \\ 0 & 8 & 26 & 30 & 6 & & \\
0 & 10 & 14 & -2 & -2 & 0 & \\
1 & 0& 0&0 &0 &0& 0
\end{array} \right],\
(f_{ij})_{0 \leq j \leq i \leq 6} = \left[ \begin{array}{ccccccc}
0 & & & & & & \\
 0 & 2 & & & & & \\
 0 & 2 & 16 & & & & \\
0 & 2 & 30 & 54 & & & \\
0 & 8 & 59 & 106 & 70 && \\
0 & 10 & 54 & 100 & 80 & 24 & \\ 1 & 6 & 15 & 20 & 15 & 6 & 1 \end{array} \right].$$}
\caption{Homology groups and $h$- and $f$-triangles for $\Delta_{D_6}$}
\label{fig:d6fh}
\end{figure}
  
\section{Exceptional Types} \label{sec:exc}

\subsection*{\texorpdfstring{$\mathbf{E_6, E_7,E_8}$}~} 

Consider the root poset of $E_6$ \cite[page 1502]{Rin18},
$E_7$ \cite[page 1503]{Rin18}, and $E_8$ \cite[page 1504]{Rin18}.
For \(n = 6, 7, 8\), let \(F_n\) be the set of all roots of rank \(n-1\) in the root poset of \(E_n\). In each case, taking \(F = F_n\) and \(d = n-1\), with \(d' = 3\) for \(n = 6\), \(d' = 4\) for \(n = 7\), and \(d' = 7\) for \(n = 8\), the conditions of \ref{lem:nonshel} are satisfied. Hence \(\Delta_{E_n}\) is not shellable for all \(n = 6, 7, 8\).

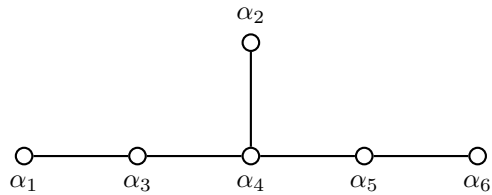
\begin{figure}[H]
\begin{tikzpicture}[
  root/.style={circle,draw,thick,inner sep=2.2pt},
  every label/.style={font=\small}
]
\node[root,label=below:$\alpha_1$] (a1) at (0,0) {};
\node[root,label=below:$\alpha_3$] (a3) at (1.5,0) {};
\node[root,label=below:$\alpha_4$] (a4) at (3.0,0) {};
\node[root,label=below:$\alpha_5$] (a5) at (4.5,0) {};
\node[root,label=below:$\alpha_6$] (a6) at (6.0,0) {};
\node[root,label=above:$\alpha_2$] (a2) at (3.0,1.5) {};

\draw[thick] (a1) -- (a3) -- (a4) -- (a5) -- (a6);
\draw[thick] (a4) -- (a2);
\end{tikzpicture}
\caption{Dynkin diagram for $E_6$} \label{fig:dye6}
\end{figure}

In \ref{fig:e6}, \ref{fig:e7} and \ref{fig:e8} we also list the 
$f$- and $h$-triangles and the homology groups of $\Delta_{E_n}$. 
 The presence of negative entries in the \(h\)-triangle and the fact that its diagonal entries do not coincide with the Betti numbers provides another proof of non-shellability.

\begin{figure}[H]
{\small
$$\begin{array}{c|cccccc} i & 0 & 1 & 2 & 3 & 4 & 5\\
\hline
\widetilde{H}_i(\Delta_{E_6}) & \ZZ^ 3 & \ZZ^{16} & \ZZ^{33} & \ZZ^ 7 & 0 & 0 
\end{array}$$

$$(h_{ij})_{0 \leq j \leq i \leq 6}  = \left[\begin{array}{ccccccc}
0 &   &     &     &     &    &  \\
0 & 3 &     &     &     &    &  \\
0 & 2 & 16  &     &     &    &  \\
0 & 6 & 30  & 33  &     &    &  \\
0 & 3 & 50  & 27 & 9  &    &  \\
0 & 16 & 6 & 12 & -6 & 2 & \\
1 & 0 & 0 & 0 & 0 & 0 & 0
\end{array}\right],$$
$$
(f_{ij})_{0 \leq j \leq i \leq 6} = \left[\begin{array}{ccccccc}
0 &   &     &     &     &    &  \\
0 & 3 &     &     &     &    &  \\
0 & 2 & 18  &     &     &    &  \\
0 & 6 & 42  & 69  &     &    &  \\
0 & 3 & 59  & 136 & 89  &    &  \\
0 & 16 & 70 & 126 & 100 & 30 & \\
1 & 6 & 15 & 20 & 15 & 6 & 1
\end{array}\right].$$}
 \caption{Homology groups and $h$- and $f$-triangles for $\Delta_{E_6}$}
\label{fig:e6}
\end{figure}
 
\begin{figure}[H]
{\small
$$\begin{array}{c|ccccccc} i & 0 & 1 & 2 & 3 & 4 & 5 & 6\\
\hline
\widetilde{H}_i(\Delta(E_7)) & \ZZ^4 & \ZZ^{29} & \ZZ^{91} & \ZZ^{68} & \ZZ^2 & 0 & 0
\end{array}$$

 $$(h_{ij})_{0 \leq j \leq i \leq 7}  = \left[\begin{array}{cccccccc}
  0 &    &     &     &     &     &   & \\ 
  0 & 4  &     &     &     &     &   & \\
  0 & 4  & 29  &     &     &     &   & \\
  0 & 4  & 71  & 91  &     &     &   & \\
  0 & 10 & 86  & 203 & 68  &     &   & \\
  0 & 10 & 109 & 128 & 34 & 4 &   & \\
  0 & 24 & 28 & -3 & 12 & -7 & 2& \\
  1 & 0  & 0  & 0 & 0 & 0 & 0 & 0 
  \end{array}\right],
 $$
 
 $$(f_{ij})_{0 \leq j \leq i \leq 7}  = \left[
 \begin{array}{cccccccc}
  0 &    &     &     &     &     &   & \\ 
  0 & 4  &     &     &     &     &   & \\
  0 & 4  & 33  &     &     &     &   & \\
  0 & 4  & 79  & 166 &     &     &   & \\
  0 & 10 & 116 & 405 & 367 &     &   & \\
  0 & 10 & 149 & 515 & 657 & 285 &   & \\
  0 & 24 & 148 & 349 & 411 & 240 & 56& \\
  1 & 7  & 21  & 35 & 35 & 21 & 7 & 1 
  \end{array}\right].$$}
 \caption{Homology groups and $h$- and $f$-triangles for $\Delta_{E_7}$}
\label{fig:e7}
\end{figure}

\begin{figure}[H]
{\small $$\begin{array}{c|cccccccc} i & 0 & 1 & 2 & 3 & 4 & 5 & 6 & 7\\
\hline
\widetilde{H}_i(\Delta_{E_8}) & \ZZ^6 & \ZZ^{59} & \ZZ^{281} & \ZZ^{454} & \ZZ^{106} & 0 & 0
& 0 
\end{array}$$

 $$ 
  (h_{ij})_{0 \leq j \leq i \leq 8}=\left[ \begin{array}{ccccccccc}
   0 & & & & & & & & \\
   0 & 6 & & & & & & & \\
   0 & 8 & 59 & & & & & &  \\
   0 & 6 & 139 & 283 & & & & & \\
   0 & 16 & 204 & 681 & 466 & & & & \\
   0 & 10 & 295 & 804 & 652 & 119 & & & \\
   0 & 24 & 279 & 467 & 125 & 26 & 5 & & \\
   0 & 42 & 56 & 17 & -12 & 15 & -8 & 2 & \\
   1 & 0  &  0 &  0 &   9 &  0 &  0 & 0 & 0
   \end{array}
   \right], $$

$$(f_{ij})_{0 \leq j \leq i \leq 8} =\left[ \begin{array}{ccccccccc}
   0 & & & & & & & & \\
   0 & 6 & & & & & & & \\
   0 & 8 & 67 & & & & & &  \\
   0 & 6 & 151 & 428 & & & & & \\
   0 & 16 & 252 & 1137 &  1367 & & & & \\
   0 & 10 & 335 & 1749 & 3185 & 1880 & & & \\
   0 & 24 & 399 & 1823 & 3440 & 2913 & 92 & & \\
   0 & 42 & 308 & 927 & 1456 & 1271 & 586 & 112 & \\
   1 & 8 & 28 & 56 & 70 & 56 & 28 & 8 & 1
   \end{array}
   \right]. $$}
 \caption{Homology groups and $h$- and $f$-triangles for $\Delta_{E_8}$}
\label{fig:e8}
\end{figure}

 \subsection*{\texorpdfstring{$\mathbf{G_2}$}~}

The root poset of $G_2$ consists of an antichain of size $2$ and a $4$-element chain whose elements all are larger than the two elements in the antichain.
It follows that $\Delta_{G_2}$ consists of one $1$-dimensional face and $4$
isolated vertices. Thus the complex is non-pure shellable
(see \ref{fig:g2} for the homology groups and $f$- and $h$-triangles).

\begin{figure}[H]
{\small $$\begin{array}{c|cc} i & 0 & 1 \\
\hline
\widetilde{H}_i(\,\Delta_{G_2}\,) & \ZZ^4 & 0 
\end{array}$$

 $$ 
  (h_{ij})_{0 \leq j \leq i \leq 2} = \left[
 \begin{array}{ccc}
 0 & & \\
 0 & 4 & \\
 1 & 0 & 0
 \end{array} \right],\
  (f_{ij})_{0 \leq j \leq i \leq 2} = \left[
 \begin{array}{ccc}
 0 & & \\
 0 & 4 & \\
 1 & 2 & 1
 \end{array} \right].$$}
 \caption{Homology groups and $h$- and $f$-triangles for $\Delta_{G_2}$}
\label{fig:g2}
 \end{figure}

 \subsection*{\texorpdfstring{$\mathbf{F_4}$}~}

Consider a choice of simple roots $\alpha_1,\alpha_2,\alpha_3,\alpha_4$
numbered from left to right in the Dynkin diagram of $F_4$ (see \ref{fig:dyf4}).

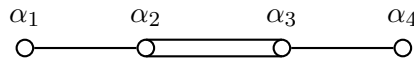
\begin{figure}[H]
\centering

\begin{tikzpicture}[
  root/.style={circle,draw,thick,inner sep=2.2pt},
  label/.style={font=\small}
]
\node[root] (a1) at (0,0) {}; 
\node[root] (a2) at (1.6,0) {}; 
\node[root] (a3) at (3.4,0) {};
\node[root] (a4) at (5.0,0) {};

\draw[thick] (a1) -- (a2);
\draw[thick] ($(a2)+(0,0.11)$) -- ($(a3)+(0,0.11)$);
\draw[thick] ($(a2)+(0,-0.11)$) -- ($(a3)+(0,-0.11)$);
\draw[thick] (a3) -- (a4);

\node [yshift=0.4cm] at (a1) {$\alpha_1$};
\node [yshift=0.4cm] at (a2) {$\alpha_2$};
\node [yshift=0.4cm] at (a3) {$\alpha_3$};
\node [yshift=0.4cm] at (a4) {$\alpha_4$};
\end{tikzpicture}
\bigskip
\caption{Dynkin diagram of $F_4$} \label{fig:dyf4}
\end{figure}

Then $\dim(\Delta_{F_4}) = 3$, and
$\alpha_1+\alpha_2+\alpha_3$, $\alpha_2+2\alpha_3$ and $\alpha_2+\alpha_3+\alpha_4$ are all roots of rank $3$ in the root poset. Hence they 
form a maximal antichain and
a $2$-dimensional facet $F$ of $\Delta_{F_4}$. 
It can be verified that
this facet satisfies the assumptions of \ref{lem:nonshel} for $d = 3$ and
$d'=2$, and hence $\Delta_{F_4}$ is not shellable.

\begin{figure}[H]
{\small 
$$\begin{array}{c|cccc} i & 0 & 1 & 2 & 3\\
\hline
\widetilde{H}_i(\,\Delta_{F_4}\,) & \ZZ^4 & \ZZ^{14} & \ZZ & 0  
\end{array}$$

$$ (h_{ij})_{0 \leq j \leq i \leq 4} = \left[\begin{array}{ccccc}
 0 &    &    &    & \\
  0 & 4  &    &    & \\
  0 & 4  & 15 &    &  \\
  0 & 12 & 6 & 2 & \\
  1 & 0 & 0 & 0 & 0 
  \end{array}\right],\
(f_{ij})_{0 \leq i \leq j \leq 4} = \left[
 \begin{array}{ccccc}
  0 &    &    &    & \\
  0 & 4  &    &    & \\
  0 & 4  & 19 &    &  \\
  0 & 12 & 30 & 20 & \\
  1 & 4 & 6 & 4 & 1 
  \end{array}\right].$$}
  \caption{Homology groups and $h$- and $f$-triangles for $\Delta_{F_4}$}
  \end{figure}

\section*{Appendix}
 \label{sec:app}
In the following, we list the $h$- and $f$-triangles of $\Delta_{\Dyck_n}^{(a,b)}$ for 
$n = 3,4,5,6$ and $(a,b) = (2,1),(3,1)$.

\begin{figure}[H]
{\small

$$ (h_{ij})_{0 \leq j \leq i \leq 2} = \left[\begin{array}{ccc}
 0 &    &     \\
  0 & 1  &    \\
  1 & 3  & 1   
  \end{array}\right],\
(f_{ij})_{0 \leq j \leq i \leq 2} = \left[
 \begin{array}{ccccc}
  0 &    &     \\
  0 & 1  &     \\
  1 & 5  & 5   
  \end{array}\right].$$

 $$ (h_{ij})_{0 \leq j \leq i \leq 3} = \left[\begin{array}{cccc}
 0 &    &   &  \\
  0 & 1  &  &  \\
  0& 2  & 5 & \\
  1 & 6 & 6 & 1 
  \end{array}\right],\
(f_{ij})_{0 \leq j \leq i \leq 3} = \left[
 \begin{array}{ccccc}
  0 &    &   &  \\
  0 & 1  &   &  \\
  0 & 2  & 7 & \\
  1 & 9 & 21 & 14 
  \end{array}\right].$$
  
$$ (h_{ij})_{0 \leq j \leq i \leq 4} = \left[\begin{array}{ccccc}
 0 &    &   &  & \\
  0 & 1  &  &  & \\
  0& 2  & 8 & & \\
  0 & 3 & 18 & 15 & \\
  1 & 19 & 20 & 10 & 1
  \end{array}\right],\
(f_{ij})_{0 \leq j \leq i \leq 4} = \left[
 \begin{array}{cccccc}
  0 &    &   &  & \\
  0 & 1  &   &  & \\
  0 & 2  & 10 & & \\
  0& 3 & 24 & 36 & \\
  1 & 14 & 56 & 84 & 42 
  \end{array}\right].$$

$$ (h_{ij})_{0 \leq j \leq i \leq 5} = \left[\begin{array}{cccccc}
 0 &    &   &     & & \\
  0 & 1  &  &     & &\\
  0& 2  & 11 &    & & \\
  0 & 3 & 27 & 40 & & \\
  1 & 4 & 42 & 84 & 35 & \\
  1 & 15 & 50 & 50 & 15 & 1 
  \end{array}\right],\
(f_{ij})_{0 \leq j \leq i \leq 5} = \left[
 \begin{array}{ccccccc}
  0 &    &   &  & & \\
  0 & 1  &   &  & & \\
  0 & 2  & 13 & & & \\
  0& 3 & 33 & 70 &  & \\
  0 & 4 & 54 & 180 & 165 & \\
  1 & 20 & 120 & 300 & 330 & 132 
  \end{array}\right].$$}

      \caption{$h$- and $f$-triangles for $\Delta^{(2,1)}_{\Dyck_n}$ and $n=3,4,5,6$}
      \label{fig:fc2}
  \end{figure}

\begin{figure}[H]

{\small

$$ (h_{ij})_{0 \leq j \leq i \leq 2} = \left[\begin{array}{ccc}
 0 &    &     \\
  0 & 1  &    \\
  1 & 8  & 12   
  \end{array}\right],\
(f_{ij})_{0 \leq j \leq i \leq 2} = \left[
 \begin{array}{ccccc}
  0 &    &     \\
  0 & 1  &     \\
  1 & 6 & 5   
  \end{array}\right].$$

 $$ (h_{ij})_{0 \leq j \leq i \leq 3} = \left[\begin{array}{cccc}
 0 &    &   &  \\
  0 & 1  &  &  \\
  0& 2  & 9 & \\
  1 & 12 & 28 & 14 
  \end{array}\right],\
(f_{ij})_{0 \leq j \leq i \leq 3} = \left[
 \begin{array}{ccccc}
  0 &    &   &  \\
  0 & 1  &   &  \\
  0 & 2  & 11 & \\
  1 & 15 & 55 & 55 
  \end{array}\right].$$

$$ (h_{ij})_{0 \leq j \leq i \leq 4} = \left[\begin{array}{ccccc}
 0 &    &   &  & \\
  0 & 1  &  &  & \\
  0& 2  & 13 & & \\
  0 & 3 & 33 & 55 & \\
  1 & 20 & 90 & 120 & 42
  \end{array}\right],\
(f_{ij})_{0 \leq j \leq i \leq 4} = \left[
 \begin{array}{cccccc}
  0 &    &   &  & \\
  0 & 1  &   &  & \\
  0 & 2  & 15 & & \\
  0& 3 & 39 & 91 & \\
  1 & 24 & 156 & 364 & 273 
  \end{array}\right].$$

$$ (h_{ij})_{0 \leq j \leq i \leq 5} = \left[\begin{array}{cccccc}
 0 &    &   &     & & \\
  0 & 1  &  &     & &\\
  0& 2  & 17 &    & & \\
  0 & 3 & 45 & 112 & & \\
  1 & 4 & 78 & 312 & 286 & \\
  1 & 30 & 220 & 550 & 495 & 132 
  \end{array}\right],\
(f_{ij})_{0 \leq j \leq i \leq 5} = \left[
 \begin{array}{ccccccc}
  0 &    &   &  & & \\
  0 & 1  &   &  & & \\
  0 & 2  & 19 & & & \\
  0& 3 & 51 & 160 &  & \\
  0 & 4 & 90 & 480 & 680 & \\
  1 & 35 & 350 & 1400 & 2380 & 1428 
  \end{array}\right].$$
  }
  \caption{$h$- and $f$-triangles for $\Delta^{(3,1)}_{\Dyck_n}$ and $n=3,4,5,6$}\label{figfc3}
  \end{figure}

\section*{Acknowledgements}
The first author was partially supported by the National Natural Science Foundation of China (Nos.\ 12671388 and 12271222).

\bibliography{ref}
\bibliographystyle{plainurl}

\end{document}